%% file: linepaper.tex
\documentclass{siamltex1213}

\input{packagesandmacrospaperline.tex}

\title{Optimal control of elliptic PDEs on surfaces of codimension 1\thanks{This work was supported by the UK Engineering and Physical Sciences Research Council (EPSRC) Grant EP/H023364/1.}}

\author{Charles Brett\thanks{Mathematics Institute, University of Warwick, Coventry, CV4 7AL, UK (\email{ceabrett@gmail.com}, \email{a.s.dedner@warwick.ac.uk}, \email{c.m.elliott@warwick.ac.uk}).}
\and
Andreas Dedner\footnotemark[2]
\and
Charles Elliott\footnotemark[2]
}

\begin{document}

%\slugger{sinum}{xxxx}{xx}{x}{x--x}%slugger should be set to mms, siap, sicomp, sicon, sidma, sima, simax, sinum, siopt, sisc, or sirev
\maketitle

\begin{abstract}
We consider an elliptic optimal control problem where the objective functional contains an integral along a surface of codimension 1, also known as a hypersurface. In particular, we use a fidelity term that encourages the state to take certain values along a curve in 2D or a surface in 3D. In the discretisation of this problem, which uses piecewise linear finite elements, we allow the hypersurface to be approximated e.g.\ by a polyhedral hypersurface. This can lead to simpler numerical methods, however it complicates the numerical analysis. We prove a priori $L^2$ error estimates for the control and present numerical results that agree with these. A comparison is also made to point control problems.
\end{abstract}



\pagestyle{myheadings}
\thispagestyle{plain}
\markboth{C. BRETT, A. S. DEDNER, and C. M. ELLIOTT}{OPTIMAL CONTROL ON SURFACES OF CODIMENSION 1}

\input{linecontrolpaper2}

\bibliographystyle{siam}
\bibliography{library}

\end{document}

%% file: packagesandmacrospaperline.tex
\usepackage{amsmath,amssymb,amsxtra,subfigure,comment,algorithm,algpseudocode,appendix,xr,comment, multirow} %or xr-hyper?
\numberwithin{equation}{section}

\newcommand{\R}{\mathbb{R}}

\newcommand{\eps}{\varepsilon}
\newcommand{\Laplace}{\Delta}

\newcommand{\rd}{\mathrm{d}}
\newcommand{\pd}{\partial}
\newcommand{\abs}[1]{\left| #1 \right|}
\newcommand{\norm}[1]{\| #1 \|}

\newtheorem{remark}[theorem]{Remark}
\newtheorem{example}[theorem]{Example}
\newtheorem{assumption}[theorem]{Assumption}

\DeclareMathOperator*{\esssup}{ess\,sup}

%% file: linecontrolpaper2.tex
\label{chap:line}

\section{Introduction}

In this paper we consider an elliptic optimal control problem with an objective functional containing an integral over a surface of codimension 1, also known as a hypersurface. In particular, we control the state to be close to prescribed values along a curve in 2D or a surface in 3D. This differs from standard elliptic optimal control problems where typically the objective functional contains the $L^2$ distance between the state and the desired state over the whole domain. So for a bounded domain $\Omega \subset \R^n$ ($n=2$ or $3$) and an $n-1$ dimensional surface $\Gamma \subset \Omega$ we consider the problem:
\begin{equation}
\min \frac{1}{2} \int_\Gamma (y-g_\Gamma)^2 \rd A + \frac{\nu}{2} \norm{\eta}^2_{L^2(\Omega)} \label{eqn:obj}
\end{equation}
subject to the state equation
\begin{equation} \label{eqn:ctosline}
\begin{aligned}
A y &= \eta \quad \text{ in } \Omega \\
y &= 0 \quad \text{ on } \partial \Omega
\end{aligned} 
\end{equation}
and the control constraints
\[
a \leq \eta \leq b.
\]
Here $g_\Gamma:\Gamma \to \R$ is the desired state on $\Gamma$, $\nu > 0$ is the cost of control, $A$ is an elliptic operator, and $a,b \in \R$ with $a<b$ are lower and upper bounds for the control. We will formulate this problem precisely using function spaces in Section \ref{sec:probform}.

The motivation for the surface fidelity term is that in some applications we may only care about the state being close to given values on a small part of the domain. Controlling the state using a distributed norm over the whole domain gives weaker control on the surface. Instead of this fidelity term we could use state constraints to force the state to take certain values, however this would lead to an optimal control with very high cost. The surface fidelity term allows for a compromise between how close the state is to the desired values on the surface and the cost of the control.

We have not seen the surface fidelity term previously used in the optimal control context in the literature, however other problems have been considered where the state is controlled on small sets. The book \cite{Troltzsch2005} formulates an optimal control problem where the objective functional is the state evaluated at a point. The paper \cite{Unger2001} considers optimally controlling the cooling of steel. This problem is formulated with an objective functional that contains the temperature of the steel at a number of points (i.e.\ point evaluations of the state) as this makes the problem more tractable. The paper \cite{BrettPoint} and thesis \cite{BrettThesis} do a detailed numerical analysis of finite element discretisations of a point control problem with an elliptic PDE state constraint. The paper \cite{Brett2013} develops an adaptive finite element method for a point control problem with a variational inequality state constraint.

In comparison to point control problems the difficulty of our problem is not the low regularity of the adjoint variable; it belongs to $H_0^1(\Omega)$ and standard literature (e.g.\ \cite{Troltzsch2005}) provides the necessary background for the analysis. The difficulty is that in order to pose a discrete problem that can be solved computationally, we may need to formulate the discrete problem with an approximation of $\Gamma$, such as a polygonal (for $n=2$) or polyhedral (for $n=3$) approximation. This complicates the numerical analysis, and estimating the error caused by approximating $\Gamma$ forces us to introduce theory usually associated with finite element methods for PDEs on hypersurfaces, such as that reviewed in \cite{Dziuk2013}. Note that we do not consider the case of a curve in 3D as this would require additional regularity of the state. In particular, we would need the state to be continuous so the problem is more closely related to that in \cite{BrettPoint}. 

Other related optimal control problems have been considered in the literature. The recent paper \cite{Gong2014} considers elliptic optimal control problems with controls on lower dimensional manifolds. Their state equation has a similar form to our adjoint equation, and our state equation has a similar form to their adjoint equation. In their discrete problems they also approximate surfaces with polyhedral surfaces. Note that our assumptions on these approximating hypersurfaces are more flexible. In papers such as \cite{Casas2012} and \cite{Pieper2013} a problem is considered where the control space is a space of measures. Supremum norm error estimates that are needed when working with state constrained elliptic optimal control problems are useful to us. The paper \cite{Leykekhman2013} proves error estimates for problems with state constraints at a finite number of points. The paper \cite{Deckelnick2007} proves error estimates for the case of global (as opposed to point) state constraints, but for a state equation with Neumann boundary conditions. A review of the analysis for standard optimal control problems can be found in \cite{Troltzsch2005} and a review of the numerical analysis can be found in \cite{Hinze2009}.

We will define an appropriate finite element discretisation of our problem and prove a priori error estimates for the $L^2(\Omega)$
error in the control. Our discretisation is based on the variational discretisation idea from \cite{Hinze2005}, as this typically allows for better error estimates. We will prove these error estimates using an approach inspired by the paper \cite{Deckelnick2007}, since we found it to be relatively simple. This will allow us to focus on the new difficulties caused by approximating $\Gamma$. We will show numerical results for $n=2$ that agree with our analytical results. We do not solve any examples for $n=3$ as the implementation would be more complicated. Table~\ref{tab:sumres} summarises our results, where $\eps>0$ is arbitrary.

\begin{table}
\centering
\begin{tabular}{l|c|c}
& $n=2$ & $n=3$ \\
\cline{1-3}
\hline
Theory & $O(h^{1-\eps})$ & $O(h^{\frac{3}{4}})$ \\
Numerics & $O(h)$ & - \\ 
\end{tabular}
\caption{The main a priori error estimates proved for $\norm{u-u_h}_{L^2(\Omega)}$.}
\label{tab:sumres}
\end{table}

In the next section we introduce some notation. In Section~\ref{sec:probform} we formulate the optimal control problem precisely and prove some analytical results. In Section~\ref{sec:disc} we introduce the theory for approximating hypersurfaces and discretise using a finite element method. In Section~\ref{sec:numanal} we prove a priori error estimates for the $L^2(\Omega)$ error in the control. In Section~\ref{sec:num} we show numerical results.

\section{Notation}
\label{sec:notationline}

% Denote by $(\cdot,\cdot)$ and $\norm{\cdot}$  the inner product and norm on $L^2(\Omega)$, and by $(\cdot, \cdot)_{L^2(\Gamma)}$ and $\norm{\cdot}_{L^2(\Gamma)}$ the inner product and norm on $L^2(\Gamma)$.

Let the domain $\Omega \subset \R^n$ ($n=2$ or $3$) be a bounded
open set. Consider the Dirichlet problem (\ref{eqn:ctosline}), where the differential operator $A$ acting on a function $z:\Omega \to \R$ is defined by
\begin{equation*}
 A z = - \sum_{i,j=1}^n \pd_{x_j}(a_{ij} \pd_{x_i} z) + a_0 z
\end{equation*}
with
\begin{align*}
&a_0 \in L^\infty(\Omega), \quad a_0(x) \geq 0 \quad \text{ for a.e. } x \in \Omega, \\ %was 
&a_{ij}=a_{ji} \in C^{0,1}(\bar{\Omega}), \\ %was 
&\exists \, \alpha>0 \text{ s.t.\ } \sum_{i,j=1}^n a_{ij}(x) \xi_i \xi_j \geq \alpha \abs{\xi}^2, \quad \forall x \in \Omega, \, \xi \in \R^n.
\end{align*}
In particular, $A=-\Laplace$ satisfies these assumptions. We want to work with a weak formulation of (\ref{eqn:ctosline}). Define the bilinear form $a:H_0^1(\Omega) \times H_0^1(\Omega) \to \R$ associated to $A$ by
\begin{align*}
a(z,v) &= \sum_{i,j=1}^n \int_\Omega a_{ij}(x) \pd_{x_i}z(x) \pd_{x_j} v(x) \rd x + \int_\Omega a_0(x) z(x) v(x) \rd x, %//&= ((a_{ij}) \nabla u, \nabla v) + (a_0 u, v)
\end{align*}
where the derivatives are taken in the weak sense. By a standard result, for $\eta \in L^2(\Omega)$ there is a unique $y \in H_0^1(\Omega)$ satisfying
\begin{equation}
a(y,v) = (\eta,v)_{L^2(\Omega)} \quad \forall v \in H_0^1(\Omega). \label{eqn:weakline}
\end{equation}

We will see shortly that it is not necessary for the state to be continuous in order for the objective functional to be well defined. The state belonging to $H_0^1(\Omega)$ is sufficient. This is why we do not make stronger assumptions on the regularity of the domain (though we will in Section~\ref{sec:disc} as it is necessary for the numerical analysis). We define the control-to-state operator $S:L^2(\Omega) \to H_0^1(\Omega)$ by $S\eta := y$, where $y$ solves (\ref{eqn:weakline}). This operator is linear, and continuous since testing (\ref{eqn:weakline}) with $v=y$ allows us to deduce
\[
\norm{S\eta}_{H_0^1(\Omega)} \leq C \norm{\eta}_{L^2(\Omega)}.
\]
Here and throughout this paper $C$ is a positive constant that may vary from line to line and is independent of the variables it precedes. This means we can define the adjoint operator $S^*:H^{-1}(\Omega) \to L^2(\Omega)$ in the usual way by
\[
( S^*z, \eta )_{L^2(\Omega)} = \langle z, S \eta \rangle_{H^{-1}(\Omega)} \quad \forall z \in H^{-1}(\Omega), \eta \in L^2(\Omega),
\]
where $\langle \cdot, \cdot \rangle_{H^{-1}(\Omega)}$ abbreviates the usual duality pairing $\langle z,v \rangle_{H^{-1}(\Omega), H_0^1(\Omega)}=z(v)$. The adjoint operator has the following property.

\begin{lemma}
\label{lem:adjreg}
For $f \in H^{-1}(\Omega)$, $p = S^* f$ if and only if $p$ satisfies
\begin{equation}
\label{eqn:adjreg}
p \in H_0^1(\Omega), \quad a(v,p) = \langle f, v \rangle_{H^{-1}(\Omega)} \quad \forall v \in H_0^1(\Omega).
\end{equation}
\end{lemma}
\begin{proof}
Suppose (\ref{eqn:adjreg}) is true. Then for any $f \in H^{-1}(\Omega)$ and $\eta \in L^2(\Omega)$, testing with $S \eta \in H_0^1(\Omega)$ we get
\[
a(S\eta,p) =  \langle f, S\eta \rangle_{H^{-1}(\Omega)}.
\]
Since $p \in H_0^1(\Omega)$, by the definition of $S$ we have 
\[
a(S \eta,p) = (\eta,p)_{L^2(\Omega)} =  (p, \eta)_{L^2(\Omega)} . 
\]
Combining these two equalities and recalling that $f$ is arbitrary we get
\[
\langle f, S\eta \rangle_{H^{-1}(\Omega)} = (p,\eta)_{L^2(\Omega)} \quad \forall f \in H^{-1}(\Omega), \eta \in L^2(\Omega). 
\]
Comparing this to the definition of the adjoint we see that $p=S^*f$. Since the adjoint operator is unique, the reverse statement must also hold. This completes the proof.
\end{proof}

Let $\Gamma$ be a $C^2$-hypersurface (see e.g.\ Section 2.2 in \cite{Dziuk2013} for this definition) for which $\Gamma \subset \Omega$ and there exists an open set $U \subset \Omega$  with a Lipschitz boundary such that $\Gamma \subset \partial U$. Note that we allow $\Gamma$ to be an open hypersurface (i.e.\ one that has a boundary) and it may have multiple connected components.

% \begin{remark}
% The hypersurface $\Gamma$ may have multiple connected components. We are able to formulate a well posed continuous optimal control problem under these assumptions. However it will be necessary to strengthen them in Section~\ref{sec:disc} in order to formulate a discrete problem for which we can prove error estimates.
% \end{remark} 

We now give an example of an admissible $\Gamma$ and a corresponding $U$.

\begin{example}
\label{ex:u}
Suppose $n=2$ and let $\Omega := (0,1)^2 \subset \R^n$. Consider
\[
\Gamma := \{ (0.25+0.5t, 0.5) : t \in (0,1) \}.
\]
i.e.\ $\Gamma$ is a straight line (an open hypersurface). Note that $\Gamma$ is orientable, so there is a continuous vector field $\mu:\Gamma \to \R^n$ such that $\mu(c)$ is a unit normal to $\Gamma$ for all $c \in \Gamma$. Therefore we can take
\[
U := \{ c+D \mu(c) \in \R^n : c \in \Gamma, 0 < D < \theta \} \subset \Omega
\]
with $\theta =0.1$. Observe that $\Gamma \subset \partial U$ and $U$ is an open set with a Lipschitz boundary. 

This construction of $U$ (for sufficiently small $\theta>0$) works for many choices of hypersurface, including closed hypersurfaces such as
\[
\Gamma := \{ (0.5 + 0.25 \cos(2 \pi t), 0.5 + 0.25 \sin(2 \pi t) : t \in (0,1) \} \subset (0,1)^2
\]
i.e.\ a circle.
\end{example}

We can define some function spaces on $\Gamma$. Denote by $C(\Gamma)$ the set of functions which are continuous on $\Gamma$. Let $L^s(\Gamma)$ with $s \in [1,\infty]$ denote the space of functions $v: \Gamma \to \R$  which are measurable with respect to the surface measure $\rd A$ (the $n-1$ dimensional Hausdorff measure) and have a finite norm
\begin{align*}
\norm{v}_{L^s(\Gamma)} &:= \left( \int_\Gamma \abs{v}^s \rd A \right)^{\frac{1}{s}} &&s \in [1,\infty), \\
\norm{v}_{L^\infty(\Omega)} &:= \esssup \abs{v}  &&p = \infty.
\end{align*}
These spaces are Banach spaces, and $L^2(\Gamma)$ is a Hilbert space with inner product
\[
(v,w)_{L^2(\Gamma)} := \int_\Gamma v w \rd A.
\]
Since $\Gamma$ is a $C^2$ hypersurface we can also define weak derivatives of functions in $L^1(\Gamma)$ and hence Sobolev spaces $H^{k,p}(\Gamma)$. As usual, $H^k(\Gamma)$ is used to denote $H^{k,2}(\Gamma)$. We do not use these Sobolev spaces directly so we leave the reader to refer to e.g.\ \cite{Dziuk2013} for the details. 
% If $\Gamma$ is a $C^2$ hypersurface (see e.g.\ Section 2.1 in \cite{Dziuk2013} for the precise definition) we can define weak derivatives of functions in $L^1(\Gamma)$ and hence also Sobolev spaces $H^{k,p}(\Gamma)$. As usual, $H^k(\Gamma)$ is used to denote $H^{k,2}(\Gamma)$. We do not use these spaces directly so we leave the reader to refer to e.g.\ \cite{Dziuk2013} for the details. 

%Later we will introduce a Lipschitz hypersurface $\Gamma_\sigma$ that approximates $\Gamma$ for which we can also define the space of continuous functions $C(\Gamma_\sigma)$ and $L^2(\Gamma_\sigma)$. 

We need to check that we can make sense of $y|_\Gamma$ for $y \in H_0^1(\Omega)$. % \begin{lemma}

\begin{lemma}
\label{lem:trace}
Let $\Gamma$ be a hypersurface satisfying the above assumptions. Then there exists a continuous linear operator $T:H_0^1(\Omega) \to L^2(\Gamma)$ such that
\begin{equation}
\label{eqn:trace1}
Ty=y|_\Gamma \quad \forall y \in H_0^1(\Omega) \cap C(\Omega).
\end{equation}
In particular,
\begin{equation}
\label{eqn:trace2}
\norm{T y}_{L^2(\Gamma)} \leq C \norm{y}_{H_0^1(\Omega)} \quad \forall y \in H_0^1(\Omega),
\end{equation}
with $C$ independent of $y$.
\end{lemma}
\begin{proof}
% Instead of $U_\delta$ consider the half strip
% \[
% U_{\delta}^+ := \{ x \in \R^{n} | 0 < d(x) < \delta \}.
% \]
% Note that $U_\delta^+$ is a Lipschitz domain in $\R^n$ with $\Gamma \subset \partial U_\delta^+$. 

Let $\tilde{T}:H^1(U) \to L^2(\partial U)$ denote the trace operator for the open set $U$ which has a Lipschitz boundary i.e.\ the unique continuous linear operator from $H^1(U)$ to $L^2(\partial U)$ such that $\tilde{T}y = y|_{\partial U}$ for all  $y \in C^{0,1}( U)$ (see e.g.\ Theorem 1.5.1.3 in \cite{Grisvard1985}). Then we can define $T:H_0^1(\Omega) \to L^2(\Gamma)$ by
\[
T y :=  \tilde{T} (y |_{U}) \big |_\Gamma.
\]

The linearity of $\tilde{T}$ implies that $T$ is linear. It is also straightforward to see that (\ref{eqn:trace1}) holds. Finally, using the continuity of the linear operator $\tilde{T}$, we get that for $y \in H_0^1(\Omega)$,
\[
\norm{Ty}_{L^2(\Gamma)} \leq \norm{\tilde{T}y}_{L^2(\partial U)} \leq C \norm{y}_{H^1(U)} \leq C \norm{y}_{H_0^1(\Omega)}
\]
with $C$ independent of $y$. So (\ref{eqn:trace2}) holds and $T$ is continuous.
\end{proof}

For $y \in H_0^1(\Omega)$ we use $y|_\Gamma$ to denote $T y$ and (when it will not cause confusion) write quantities such as $\norm{y}_{L^2(\Gamma)}$ instead of $\norm{y|_\Gamma}_{L^2(\Gamma)}$. So with this notation the objective functional (\ref{eqn:obj}) is well defined. 

\section{Problem formulation}
\label{sec:probform}

We now formulate the optimal control problem precisely as
\begin{align*}
\text{min } &J(y,\eta) := \frac{1}{2} \norm{y-g_\Gamma}^2_{L^2(\Gamma)} +  \frac{\nu}{2} \norm{\eta}^2_{L^2(\Omega)} \\
\text{over } & H_0^1(\Omega) \times L^2(\Omega) \\
\text{s.t.\ } & y=Su \text{ (i.e.\ (\ref{eqn:weakline}) holds}) \\
\text{and } & \eta \in U_{ad} := \{ \eta \in L^2(\Omega) : a \leq \eta \leq
b \text{ a.e.\ in } \Omega \}.
\end{align*}
Or equivalently, define the reduced objective functional $\hat{J}:H_0^1(\Omega) \to \R$ by $\hat{J}(\eta) = J(S \eta, \eta)$ and consider the optimisation problem
\begin{equation} \label{eqn:controlprobline}
\min \hat{J}(\eta) \text{ over } U_{ad}.
\end{equation} 
Here $g_\Gamma \in L^2(\Gamma)$, $a,b \in \R$ with either $a<b$ or $b=-a=\infty$, and $\nu>0$.

\begin{theorem}
\label{thm:existline}
Problem (\ref{eqn:controlprobline}) has a unique solution $u \in U_{ad}$.
\end{theorem}
\begin{proof}
This result follows using the same argument as is used for proving existence and uniqueness of solutions to standard optimal control problems. %, so we only outline the main ideas. 
See e.g.\ Theorem~2.14 in \cite{Troltzsch2005} for the details.

%As $\hat{J} \geq 0$ we can construct an infimising sequence $\{ \eta_{n} \} \subset U_{ad}$ i.e.\ a sequence such that $\hat{J}(\eta_n) \to \inf_{\eta \in U_{ad}} \hat{J}(\eta)$. Note that $U_{ad}$ is a nonempty, closed, bounded and convex subset of a real reflexive Banach space, so it is weakly sequentially compact. This means there is a subsequence $\{\eta_{n_k}\}$ converging to some $u \in U_{ad}$. Since $S:L^2(\Omega) \to H_0^1(\Omega)$ is continuous, $\hat{J}$ is continuous. $\hat{J}$ is also convex, so it is weakly lower semicontinuous. Therefore $u$ achieves the infimum of $\hat{J}$ i.e.\ it is a minimiser of $\hat{J}$. By a contradiction argument the strict convexity of $\hat{J}$ gives that $u$ is the unique minimiser
\end{proof}

\begin{theorem}
\label{thm:opt}
$u \in L^2(\Omega)$ is a solution of (\ref{eqn:controlprobline})
% with corresponding state $y \in H_0^1(\Omega)$ 
if and only if there exist $p \in H_0^1(\Omega)$ such that
\begin{subequations}
\begin{align}
&u \in U_{ad}, \quad ({p} + \nu {u}, v -{u} )_{L^2(\Omega)} \geq 0  &&\forall v \in U_{ad}, \label{eqn:sysa} \\
%&a({y},v)=({u},v) &&\forall v \in H_0^1(\Omega) \\
&a({p},v)= \int_\Gamma (S u-g_\Gamma)v \rd A&&\forall v \in H_0^1(\Omega) \label{eqn:adjoint}.
\end{align}
\end{subequations}
\end{theorem}
\begin{proof}
$\hat{J}:L^2(\Omega) \to \R$ has a G\^{a}teaux derivative $\hat{J}':L^2(\Omega) \to L^2(\Omega)^*$ and is (strictly) convex, so ${u}$ is a solution of (\ref{eqn:controlprobline}) iff
\begin{equation}
\label{eqn:vi}
u \in U_{ad},\quad \langle \hat{J}'(u), v-u \rangle_{L^2(\Omega)^*} \geq 0 \quad \forall v \in U_{ad}.
\end{equation}
Calculating $\hat{J}'(u)$ we see that (\ref{eqn:vi}) becomes
\begin{equation}
\label{eqn:deriv}
u \in U_{ad}, \quad \int_\Gamma (S{u}-g_\Gamma)S(v-{u}) \rd A + \nu ({u},v-{u})_{L^2(\Omega)} \geq 0 \quad \forall v \in U_{ad}.
\end{equation}
Let $f_u(v) := \int_\Gamma (S u-g_\Gamma)v$, so $f_u \in H^{-1}(\Omega)$. Take ${p}=S^* f_{{u}}$, then by Lemma~\ref{lem:adjreg} we have $p \in H_0^1(\Omega)$ and it satisfies (\ref{eqn:adjoint}). Also
\[
( {p}, v-{u})_{L^2(\Omega)}  = ( S^* f_{{u}}, v-{u})_{L^2(\Omega)} = \langle f_{{u}}, S(v-{u}) \rangle_{H^{-1}(\Omega)} = \int_\Gamma
(S{u}-g_\Gamma) S(v-{u}) \rd A.
\]
Therefore (\ref{eqn:deriv}) is equivalent to (\ref{eqn:sysa}), which proves the result.
\end{proof}

\begin{corollary}
\label{cor:reg2}
If $u \in U_{ad}$ is a solution of (\ref{eqn:controlprobline}) then it has the additional regularity that $u \in H_0^1(\Omega)$. 
\end{corollary}
\begin{proof}
Observe that (\ref{eqn:sysa}) is equivalent to
\begin{equation}
\label{eqn:proj}
u(x)=\mathbb{P}_{[a,b]} \left(-\frac{1}{\nu} p(x) \right) \quad \text{ for a.e. } x \in \Omega, 
\end{equation}
where $\mathbb{P}_{[a,b]}(v) := v + \max (0, a - v ) - \max (0, v - b )$. If $v,w \in H_0^1(\Omega)$ then $\max (v,w) \in H_0^1(\Omega)$ (see e.g.\ \cite{Morrey1966}). So since $p \in H_0^1(\Omega)$ we also get this additional regularity for $u$.
\end{proof}

% If we wanted we could deduce the additional regularity $p \in H_0^1(\Omega)$ but, unlike in the previous chapter, this is not necessary.

\begin{remark}
By increasing the relative weight given to the fidelity term this problem could be related to an optimal control problem with state constraints. We do not include any results on this. %, but a result similar to Theorem \ref{point-thm:link} could be proved.
\end{remark}

\section{Discretisation}
\label{sec:disc}

In this section we will formulate a discrete problem that uses the variational discretisation idea from \cite{Hinze2005}.

After replacing the control-to-state operator with a discrete control-to-state operator we are still left with a discrete problem that may be hard to solve computationally. For example, if $\Gamma$ has a complicated form then it may be difficult to calculate integrals of functions defined over $\Gamma$, making the implementation of a standard finite element method impractical. Therefore we will allow in our discretisation the approximation of $\Gamma$ with another hypersurface $\Gamma_\sigma$. If chosen carefully this may simplify the calculations needed for a numerical method, but still allow us to prove the same error estimates. In particular, we want to consider taking $\Gamma_\sigma$ to be a polyhedral interpolation of $\Gamma$. In order to formulate such a discrete problem we now make stronger assumptions on $\Omega$, $A$ and $\Gamma$ than were necessary to pose the continuous problem (\ref{eqn:controlprobline}). 

From now onwards suppose that $\Omega$ is convex with a $C^2$ boundary. The assumption of convexity simplifies the presentation by ensuring that the finite element space for the state (defined shortly) is a subset of $C_0(\Omega)$. Also from now onwards assume that the boundary of $\Omega$ and the coefficient functions $a_{ij}$ and $a_0$ in the elliptic operator $A$ are sufficiently smooth that for $2 \leq s < \infty$,
\begin{align}
\norm{S \eta}_{W^{2,s}(\Omega)} &\leq C(s) \norm{\eta}_{L^s(\Omega)} \quad \forall \eta \in L^s(\Omega) \label{eqn:boundpoint} \\
\norm{S \eta}_{W^{1,\infty}(\Omega)} &\leq C\norm{\eta}_{H^1(\Omega)} \quad \quad \forall \eta \in H^1(\Omega). \label{eqn:boundpoint2}
\end{align}
This holds, for example, when $\Omega$ has a $C^3$ boundary and $a_{ij},a_0 \in C^2(\bar{\Omega})$ (see e.g.\ Theorems~9.14 and 9.15 in \cite{Trudinger} and Theorem~5 in Section~6.3 in \cite{Evans}).
%(for a proof of the regularity and stability result see Theorems~2.2.2.3 and 3.2.1.2 in \cite{Grisvard1985}). 

In addition to the assumptions in Section~\ref{sec:notationline}, suppose that $\Gamma$ is orientable. This means that $\Gamma$ has a unit normal vector field $\mu$ that is continuous (see Example~\ref{ex:u}), allowing us to construct the one sided strip %and $\partial \Gamma=\emptyset$ for $n=3$ (i.e.\ only consider closed $\Gamma$ in this dimension). 
\[
U_\delta := \{ c + D \mu(c) \in \R^n : c \in \Gamma, -\delta < D < \delta \}.
\]
Since $\Gamma$ is $C^2$ there exists a $\delta > 0$ such that for each $x \in U_\delta$ there is a unique $c(x) \in \Gamma$ and some $- \delta < D(x) < \delta$ satisfying 
\begin{equation}
\label{eqn:opop}
x=c(x) + D(x) \mu (c (x))
\end{equation}
(see Lemma~2.8 in \cite{Dziuk2013}). We call $D(x): U_\delta \to \R$ a signed distance function for $\Gamma$, and it makes sense for both open and closed $\Gamma$. When $\Gamma$ is closed it agrees with the usual definition of the signed distance function $d$ on $U_\delta$ (see e.g.\ page 296 in \cite{Dziuk2013}). Therefore all the results we need from \cite{Dziuk2013} that are proved using $d$ also hold for our possibly open hypersurfaces using $D$. %We will later use (\ref{eqn:opop}) to map between points on an approximating hypersurface and points on $\Gamma$.
% Let $d$ denote the oriented distance function for $\Gamma$ with unit normal $\mu$ (see Section 2.3 in \cite{Dziuk2013}). Since $\Gamma$ is $C^2$, for some $\delta>0$ there exists a strip
% \[
% U_\delta := \{ x \in \R^{n} | \abs{d(x)} < \delta \}
% \]
% around $\Gamma$ such that $U_\delta \subset \Omega$ and for every point $x \in U_\delta$ there is a unique point $c(x) \in \Gamma$ such that
% \begin{equation}

% x = c(x) + d(x) \mu (c(x)).
% \end{equation}

% This regularity on $\Gamma$ is sufficient to pose the continuous optimal control problem, but we will make stronger assumptions in Section~\ref{sec:disc} in order to enable the numerical analysis.
% %the stronger assumption that $\Gamma \in C^2$, and also $\partial \Gamma = \emptyset$ (i.e.\ $\Gamma$ is closed) in the case $n=3$.
% Note that we currently do not assume that $\Gamma$ is connected, so it may have multiple components.

% \begin{remark}
% We do not consider open $\Gamma$ for $n=3$ as this avoids some technicalities in the numerical analysis: Approximating closed hypersurfaces is standard (see e.g.\ \cite{Dziuk2013}). When $n=2$ the theory associated with this can be applied to open hypersurfaces in an obvious way, 
% %(and we will not always draw attention to the details), 
% but for $n=3$ it does not apply automatically. In particular, the lift operator that we define later in this section would need to be different.
% \end{remark}

Let $\Gamma_\sigma$ be a family of Lipschitz hypersurfaces contained in $U_\delta \cap \Omega$ that are indexed by the parameter $\sigma>0$. We intend this family of hypersurfaces to increasingly well approximate $\Gamma$ as $\sigma \to 0$. We suppose that they satisfy a covering condition for each connected component $\Gamma$; for each $c \in \Gamma$ there is a unique $x \in \Gamma_\sigma$ with $c=c(x)$, where $c(x)$ is defined by (\ref{eqn:opop}). Two possible constructions of $\Gamma_\sigma$ that we will later consider are:
\begin{itemize}
\item $\Gamma_\sigma := \Gamma$ for all $\sigma>0$ i.e.\ we do not approximate $\Gamma$;
\item $\Gamma_\sigma$ is the union of finitely many closed $(n-1)$-simplices with maximum diameter $\sigma$. In particular, we will suppose $\Gamma_\sigma$ is a polygonal or polyhedral interpolation of $\Gamma$. Note that such $\Gamma_\sigma$ will always violate the covering condition when $n=3$ unless $\Gamma$ has a polygonal boundary.
\end{itemize}
Since the $\Gamma_\sigma$ are Lipschitz we can define the function spaces $L^2(\Gamma_\sigma)$ and $C(\Gamma_\sigma)$ in the same way as for $\Gamma$.

% \begin{remark}
% \label{rem:conn}
% Our theory also applies when $\Gamma$ is not connected. In this case we just treat each connected component separately.  %, unless we make further assumptions. % This will break our and the lift operator (see Section~\ref{sec:disc}). Under suitable restrictions, the results we prove for connected $\Gamma$ would also hold for ones with multiple components. 
% % allowing for control on interesting surfaces to be considered. For example, see Figure~\ref{fig:spoke}, where $\Gamma$ is formulated as 6 open lines originating from the point $(0.5,0.5)$. 
% \end{remark}

We can take a family of polygonal or polyhedral approximations $\Omega_h \subset \Omega$ such that the corners of $\Omega_h$ lie on the boundary of $\Omega$ and $\abs{\Omega \setminus \Omega_h} \leq C h^2$. On each $\Omega_h$ we can construct a conforming triangulation $T_h$ of triangles or tetrahedra $T$ with maximum diameter $h:=\max_{T \in T_h} h(T)$, where $h(T)$ is the diameter of an element $T$. Additionally suppose that the family of triangulations are conforming and quasi-uniform i.e.\ there exists a constant $C$ such that
\[
\frac{h(T)}{\rho(T)} \leq C \quad \forall T \in T_h,
\]
where $\rho(T)$ is the radius of the largest ball contained in $T$, and there exists a constant $C$ such that
\[
\frac{h}{h(T)} \leq C \quad \forall T \in T_h
\]
(see e.g.\ Chapter 3 in \cite{Ciarlet1978}). We can define the following family of discrete spaces of piecewise linear globally continuous finite elements which vanish on the boundary:
\begin{align*}
V_h := \{v_h \in C_0(\Omega) : v_h|_T \in P_1(T) \text{ for all } T \in T_h \text{ and } v_h|_{\Omega \setminus \Omega_h} = 0 \}.
\end{align*}
Here $P_1(T)$ is the set of affine functions over $T$. We use this to define a discrete approximation to $S$. For $\eta \in L^2(\Omega)$ let $y_h$ be the unique function in satisfying 
\[
y_h \in V_h, \quad a(y_h,v_h) = (\eta,v_h) \quad \forall v_h \in V_h,
\]
and define $S_h : L^2(\Omega) \to V_h$ by $S_h\eta=y_h$. Observe that $V_h$ is a finite dimensional subspace of $H_0^1(\Omega)$, so it a Banach space when equipped with the $H_0^1(\Omega)$ norm. Therefore $S_h$ is a linear and continuous operator between Banach spaces and we are able to define an adjoint operator. 

\begin{remark}
\label{rem:cmvb}
We choose the range of $S_h$ to be $V_h \subset C(\bar{\Omega})$ rather than the range of $S$ so that $S_h \eta |_{\Gamma_\sigma}$ is well defined and belongs to $L^2(\Gamma_\sigma)$; Lemma~\ref{lem:trace} does not apply since we assume that $\Gamma_\sigma$ is Lipschitz rather than $C^2$.
\end{remark}

% This operator is linear and continuous, since testing the above equation with $v_h=y_h$ allows us to deduce
% \begin{equation*}
% %\label{eqn:disccts}
% \norm{y_h}_{H_0^1(\Omega)} \leq C \norm{\eta}_{L^2(\Omega)}. 
% \end{equation*}
% So $S_h$ has an adjoint operator denoted by $S_h^*$ with the property that if $f \in H^{-1}(\Omega)$ then $p_h=S_h^* f$ if and only if $p_h$ satisfies
% \[
% p_h \in V_h, \quad a(v_h,p_h) = \langle f, v_h \rangle_{H^{-1}(\Omega)} \quad \forall v_h \in V_h. 
% \]
% This is proved in the same way as Lemma~\ref{lem:adjreg}.

For $\eta \in L^2(\Omega)$ we have by (\ref{eqn:boundpoint}) and a Sobolev embedding result that $S \eta \in C(\bar{\Omega})$, so it makes sense to look at $\norm{S \eta-S_h \eta}_\infty$, where $\norm{\cdot}_\infty$ denotes the supremum norm.

\begin{lemma}
\label{lem:prelim}
For $\eta \in L^s(\Omega)$ and $2\leq s<\infty$,
\begin{equation}
%\label{eqn:errinf}
\norm{S\eta-S_h \eta}_\infty \leq C(s)h^{2-\frac{n}{s}} \norm{\eta}_{L^s(\Omega)}, \quad n=2,3.
\end{equation}
\end{lemma}
\begin{proof}
See Lemma~4.1 in \cite{BrettPoint}.
\end{proof}

\begin{corollary}
\label{cor:prelim}
For $\eta \in L^2(\Omega)$,
\begin{equation}
\label{eqn:errinf}
\norm{S\eta-S_h \eta}_\infty \leq Ch^{2-\frac{n}{2}} \norm{\eta}_{L^2(\Omega)}, \quad n=2,3.
\end{equation}
For $\eta \in H^1(\Omega)$,
\begin{equation}
\label{eqn:lineinfest}
\norm{S\eta-S_h \eta}_{L^\infty(\Omega_h)} \leq 
\begin{cases}C(\eps) h^{2-\eps}\norm{ \eta }_{H^1(\Omega)} & n=2, \\
C h^{\frac{3}{2}}\norm{ \eta }_{H^1(\Omega)} &n=3. \\
\end{cases}
\end{equation}
\end{corollary}
\begin{proof}
The first estimate follows by taking $s=2$ in Lemma~\ref{lem:prelim}. The other estimate follow by combining the lemma with Sobolev embedding results.

If $\eta \in H^1(\Omega)$ then
\[
H^1(\Omega) \hookrightarrow L^s(\Omega) \quad \forall s \in 
\begin{cases} [1,\infty) & n=2, \\
[1,6] & n=3.
\end{cases}
\]
So taking $s$ sufficiently large for $n=2$ and $s=6$ for $n=3$ we get that for any $\eps>0$,
\[
C(s) h^{2-\frac{n}{s}} \norm{ \eta }_{L^s(\Omega)} \leq 
\begin{cases}C(\eps) h^{2-\eps}\norm{ \eta }_{H^1(\Omega)} & n=2, \\
C h^{\frac{3}{2}}\norm{ \eta }_{H^1(\Omega)} &n=3. \\
\end{cases}
\]
Combining this with Lemma~\ref{lem:prelim} gives the result.
\end{proof}

\begin{comment}
\begin{corollary}
\label{cor:prelim}
For $\eta \in L^2(\Omega)$,
\begin{equation}
\label{eqn:errinf}
\norm{S\eta-S_h \eta}_\infty \leq Ch^{2-\frac{n}{2}} \norm{\eta}_{L^2(\Omega)}, \quad n=2,3.
\end{equation}
For $\eta \in H^1(\Omega)$ and any $\eps>0$,
\begin{equation}
\label{eqn:lineinfest}
\norm{S \eta - S_h \eta}_\infty \leq \norm{ \eta }_{H^1(\Omega)}
\begin{cases}C(\eps) h^{2-\eps} & n=2, \\
C h^{\frac{3}{2}} &n=3. \\
\end{cases}
\end{equation}
\end{corollary}
\begin{proof}
The first estimate follows by taking $s=2$ in Lemma~\ref{lem:prelim}. The second estimate follows by combining the lemma with Sobolev embedding results. In particular, if $\eta \in H^1(\Omega)$ then
\[
H^1(\Omega) \hookrightarrow L^s(\Omega) \quad \forall s \in 
\begin{cases} [1,\infty) & n=2, \\
[1,6] & n=3.
\end{cases}
\]
So by taking $s$ sufficiently large for $n=2$ and $s=6$ for $n=3$ we get: For any $\eps>0$,
\[
C(s)h^{2-\frac{n}{s}}\norm{\eta}_{L^{s}(\Omega)} \leq 
\begin{cases}C(\eps) h^{2-\eps}\norm{ \eta }_{H^1(\Omega)} & n=2, \\
C h^{\frac{3}{2}}\norm{ \eta }_{H^1(\Omega)} &n=3. \\
\end{cases}
\]
This proves the third estimate.
\end{proof}
\end{comment}

We are now ready to introduce the discrete problem. Define the discrete reduced objective functional $\hat{J}_h: L^2(\Omega) \to \R$ by
\[
\hat{J}_h(\eta) := \frac{1}{2} \norm{ S_h \eta - g_{\Gamma,\sigma} }^2_{\Gamma_\sigma} +\frac{\nu}{2} \norm{\eta}^2_{L^2(\Omega)}
\]
%$J_h: C(\bar{\Omega}) \times L^2(\Omega) \to \R$ by%\begin{equation*}
%J_h(y,\eta) := \frac{1}{2} \norm{ y - g_{\Gamma,\sigma} }^2_{\Gamma_\sigma} +\frac{\nu}{2} \norm{\eta}^2_{L^2(\Omega)},
%\end{equation*}
and consider the following discrete problem based on the variational discretisation concept from \cite{Hinze2005}:
\begin{equation}
\label{eqn:controlprob2} 
\min \hat{J}_h(\eta) \text{ over } \eta \in U_{ad}.
\end{equation}
Here $g_{\Gamma,\sigma} \in L^2(\Gamma_\sigma)$ is a function that will be defined to approximate $g_\Gamma \in L^2(\Gamma)$. Also let the norm $\norm{\cdot}_{\Gamma_\sigma} := \sqrt{ m_\sigma(\cdot,\cdot)}$, where $m_\sigma:L^2(\Gamma_\sigma) \times L^2(\Gamma_\sigma) \to \R$ is some inner product that will be defined to approximate the $L^2(\Gamma)$ inner product. Note that the restriction of $S_h \eta$ to $L^2(\Gamma_\sigma)$ is well defined by Remark~\ref{rem:cmvb}. The assumptions we have made so far on $\Gamma_\sigma$, $g_{\Gamma,\sigma}$ and $m_\sigma$ are sufficient to prove existence of a solution to (\ref{eqn:controlprob2}) and derive optimality conditions. These solutions will not necessarily closely approximate the solution of the continuous problem (\ref{eqn:controlprobline}), but we will impose further assumptions in the next section which ensure this.

%\begin{remark}
% We introduce the space $Y$ to allow for different numerical methods to be used depending on the regularity of $g_\Gamma$. For example, if we only know $g_\Gamma \in L^2(\Omega)$ we may take $Y=L^2(\Gamma_h)$ and $m_h$ to be the $L^2(\Gamma_h)$ inner product (assuming we can compute this). If $g_\Gamma$ is smooth then we may take $Y=C^\infty(\Omega)$ and define $m_h$ to be a numerical quadrature. This would require point evaluations but avoids the need for exact integration. We will explain some examples in more detail at the end of the next section. Our results will only require abstract assumptions.
% \end{remark}

\begin{theorem} Problem (\ref{eqn:controlprob2}) has a unique solution ${u}_h \in U_{ad}$. Moreover, ${u}_h \in L^2(\Omega)$ is a solution of (\ref{eqn:controlprob2}) if and only if there exists ${p}_h \in V_h$ such that
\begin{subequations}
\begin{align}
&u_h \in U_{ad}, \quad (p_h + \nu u_h,v-u_h) \geq 0  && \forall v \in U_{ad} \label{eqn:system2a} \\
%a(y_h,v_h) = (u_h,v_h) \forall v_h \in V_h \\
&a(v_h,p_h) = m_\sigma (S_h u_h|_{\Gamma_\sigma}-g_{\Gamma,\sigma}, v_h|_{\Gamma_\sigma})  && \forall v_h \in V_h. \label{eqn:system2b} 
\end{align}
\end{subequations}
\end{theorem}
\begin{proof}
The proof follows by the same arguments as in Theorem~\ref{thm:opt}.
\end{proof}

We are minimising over the infinite dimensional space $U_{ad}$, but (\ref{eqn:system2a}) implies
\[
u_h = \mathbb{P}_{[a,b]} \Big (-\frac{1}{\nu}p_h \Big ) \quad \text{a.e.\ in } \Omega.
\]
So the control is implicitly discretised through $S_h$. This means that the above optimality conditions can be solved computationally for appropriate choices of $\Gamma_\sigma$, $m_\sigma$ and $g_{\Gamma,\sigma}$. 

\section{Numerical analysis}
\label{sec:numanal}

In this section we will prove an a priori $L^2(\Omega)$ error estimate for convergence of the discrete optimal control problem (\ref{eqn:controlprob2}) to the continuous optimal control problem (\ref{eqn:controlprobline}). This will require additional assumptions on $\Gamma_\sigma$, $m_\sigma$ and $g_{\Gamma,\sigma}$. In order to write down these assumptions we need a way to compare functions defined on $\Gamma_\sigma$ with functions defined on $\Gamma$. For this purpose we introduce the lift operator (see e.g.\ Section~4.1 in \cite{Dziuk2013} for more details). %Note that the lift operator is used in our numerical analysis but not in the statement of the discrete problem or our numerical method for solving it.

Let $w_\sigma$ be a function defined on $\Gamma_\sigma$. Due to the covering condition, for each $c \in \Gamma$ there is a unique $x \in \Gamma_\sigma$ with $c=c(x)$ (see (\ref{eqn:opop}). We will denote this $x$ by $x(c)$. Then the lift operator $(\cdot)^l$ mapping a function defined on $\Gamma_\sigma$ to a function defined on $\Gamma$ is given by
\[
w_\sigma^l (c) := w_\sigma(x(c)) \quad \forall c \in \Gamma.
\]
% \[
% x = c + d(x)\mu(c).
% \]
Note that the inverse lift operator $(\cdot)^{-l} := ((\cdot)^l)^{-1}$ is well defined. We also use $x(c)$ to define the distance $D_\sigma : \Gamma \mapsto \R$ between $\Gamma$ and $\Gamma_\sigma$ by
\[
D_\sigma(c) := \abs{c-x(c)}.
\]

We can now impose the following additional assumptions on $\Gamma_\sigma$, $m_\sigma$ and $g_{\Gamma,\sigma}$ (see Section~\ref{sec:disc} for the previous assumptions). 

\begin{assumption}
\label{ass:ass}
$\Gamma_\sigma$, $m_\sigma$ and $g_{\Gamma,\sigma}$ satisfy
\begin{align}
\sup_{c \in \Gamma} D_\sigma(c) &\leq C \sigma^2, \label{eqn:dcond} \\
\abs{ (w_\sigma^l,z_\sigma^l)_{L^2(\Gamma)} - m_\sigma(w_\sigma,z_\sigma) } &\leq C \sigma^2 \norm{w_\sigma^l}_{L^2(\Gamma)} \norm{z_\sigma^l}_{L^2(\Gamma)} \quad \forall w_\sigma,z_\sigma \in L^2(\Gamma_\sigma) \label{eqn:mcond}, \\ %Y \cup C(\Gamma_\sigma), \label{eqn:mcond} \\ 
\norm{g_\Gamma - g_{\Gamma,\sigma}^l}_{L^2(\Gamma)} &\leq C \sigma^2 \label{eqn:ygcond}
\end{align}
with $C$ independent of $\sigma$.
\end{assumption}

Under these assumptions we can prove some lemmas which will enable us to prove a priori $L^2(\Omega)$ error estimates for the control.

\begin{lemma}
\label{lem:h01bound}
Let $u_h$ be the solution of (\ref{eqn:controlprob2}). For sufficiently small $\sigma$ and $h$,
\begin{equation*}
\norm { u_h }_{H^1(\Omega)} + \norm { p_h }_{H^1(\Omega)} \leq Ch^{-\frac{n}{2}} \sigma^2   
\end{equation*}
with $C$ independent of $\sigma$ and $h$.
\end{lemma}
\begin{proof}
Fix $\sigma_0$, $h_0 \in \R$ and suppose throughout this proof that $0<\sigma<\sigma_0$ and $0<h<h_0$.

First observe that $\norm { u_h }_{H_0^1(\Omega)} \leq C \norm { p_h }_{H_0^1(\Omega)}$, since (\ref{eqn:system2a}) gives $u_h=\mathbb{P}_{[a,b]} ( -\frac{1}{\nu}p_h )$. Also, for $v \in H_0^1(\Omega)$ we have by the Poincar\'{e} inequality that $\norm{v}_{H^1(\Omega)} \leq C\norm{v}_{H_0^1(\Omega)}.$ So we just need to show that $\norm { p_h }_{H_0^1(\Omega)} \leq C$. 

Testing (\ref{eqn:system2b}) with $v_h=p_h$ and using the coercivity of $a(\cdot,\cdot)$ and the boundedness of $m_\sigma(\cdot,\cdot)$ we get that
\begin{align}
\alpha \norm { p_h }^2_{H_0^1(\Omega)} &\leq C \norm{S_h u_h - g_{\Gamma, \sigma}}_{L^2(\Gamma_\sigma)} \norm{p_h}_{L^2(\Gamma_\sigma)} \nonumber\\
&\leq C (\norm{S_h u_h}_{L^2(\Gamma_\sigma)} + \norm{g_{\Gamma,\sigma}^l}_{L^2(\Gamma)}) \norm{p_h^l}_{L^2(\Gamma)} \nonumber \\
&\leq C (\norm{S_h u_h}_{L^\infty(\Omega)} + 1)( \norm{p_h}_{L^2(\Gamma)} + \norm{p_h^l-p_h}_{L^2(\Gamma)} ) \nonumber \\
&\leq C (\norm{u_h}_{L^2(\Omega)} + 1)( \norm{p_h}_{H_0^1(\Omega)} + \norm{p_h^l-p_h}_{L^2(\Gamma)} ) \nonumber \\
& \leq C ( \norm{p_h}_{H_0^1(\Omega)} + \norm{p_h^l-p_h}_{L^2(\Gamma)} ). \label{eqn:zkzk}
\end{align}
For the third inequality we have used assumption (\ref{eqn:ygcond}). For the fourth inequality have used the supremum norm error estimate (\ref{eqn:errinf}) and the trace inequality from Lemma~\ref{lem:trace}. For the last inequality we have used that
\[
\frac{\nu}{2}\norm{ u_h }^2_{L^2(\Omega)} \leq \hat{J}_h(u_h) = \hat{J}_h(0) = \norm{g_{\Gamma,\sigma}}^2_{\Gamma_\sigma} \leq \norm{g_{\Gamma,\sigma}^l}^2_{L^2(\Gamma)} \leq C,
\]
which implies $\norm{ u_h }^2_{L^2(\Omega)} \leq C$.

Note that for $v \in W^{1,\infty}(\Omega)$ and $x_1,x_2 \in \Omega$,
\[
\abs{v(x_1)-v(x_2)} \leq \norm{\nabla v}_{L^\infty(\Omega)} \abs{x_1-x_2},
\] 
(see Theorem~2.1.4 in \cite{Ziemer}) so
\begin{align}
\norm{v^l-v}_{L^2(\Gamma)} &= \left( \int_\Gamma \big (v(c)-v(x(c)) \big)^2 \rd c \right )^{\frac{1}{2}} \nonumber \\
&\leq \norm{\nabla v}_{L^\infty(\Omega)} \sup_{c \in \Gamma} \abs{c-x(c)} \nonumber \\
&\leq \norm{v}_{W^{1,\infty}(\Omega)} \sup_{c \in \Gamma} D_\sigma(c). \label{eqn:jjscov}
\end{align}
Using this with $v=p_h$, an inverse inequality, and assumption (\ref{eqn:dcond}) we get
\begin{align*}
\norm{p_h^l-p_h}_{L^2(\Gamma)} \leq C h^{-\frac{n}{2}} \sigma^2 \norm{v_h}_{H_0^1(\Omega)}.
\end{align*}
Combining this with (\ref{eqn:zkzk}) gives
\[
\alpha \norm { p_h }^2_{H_0^1(\Omega)} \leq C \norm{p_h}_{H_0^1(\Omega)}( 1 +  h^{-\frac{n}{2}} \sigma^2),
\]
so the result follows.
\end{proof}

\begin{lemma} \label{lem:ndiff}
For some $\eta \in H_0^1(\Omega)$ set $w:=S\eta|_\Gamma-g_\Gamma$ and $w_\sigma:=S_h \eta|_{\Gamma_\sigma} - g_{\Gamma,\sigma}$. Then for sufficiently small $\sigma$ and $h$,
\begin{equation*}
\abs{ \norm{w}^2_{L^2(\Gamma)}-\norm{w_\sigma}^2_{\Gamma_\sigma} } \leq C( \norm{\eta}_{H^1(\Omega)} ) \left ( \norm{ S \eta - S_h \eta}_\infty  + \sigma^2 \right )
\end{equation*}
with $C$ independent of $\sigma$ and $h$. 
\end{lemma}
\begin{proof}
Fix $\sigma_0$, $h_0 \in \R$ and suppose throughout this proof that $0<\sigma<\sigma_0$ and $0<h<h_0$.

Make the splitting
\[
\abs{ \norm{w}^2_{L^2(\Gamma)}-\norm{w_\sigma}_{\Gamma_\sigma}^2 }  \leq \abs{ \norm{w}^2_{L^2(\Gamma)}- \norm{w_\sigma^l}^2_{L^2(\Gamma)} } + \abs{ \norm{w_\sigma^l}^2_{L^2(\Gamma)} - \norm{w_\sigma}^2_{\Gamma_\sigma} }. 
\]
To bound the first term on the right hand side note that
\begin{align}
\abs{ \norm{w}^2_{L^2(\Gamma)} - \norm{w_\sigma^l}^2_{L^2(\Gamma)} } =& \abs{ (w+w_\sigma^l,w-w_\sigma^l) } \nonumber \\
\leq&  \norm{w+w_\sigma^l}_{L^2(\Gamma)} \norm{w-w_\sigma^l}_{L^2(\Gamma)} \nonumber \\
\leq& \left ( \norm{w}_{L^2(\Gamma)} + \norm{w_\sigma^l}_{L^2(\Gamma)} \right ) \norm{ w - w_\sigma^l }_{L^2(\Gamma)}. \label{eqn:lkjlkj}
\end{align}
Using the trace result from Lemma~\ref{lem:trace} and the continuity of $S$ we have
\begin{align*}
\norm{w}_{L^2(\Gamma)} &= \norm{S \eta - g_\Gamma}_{L^2(\Gamma)} \\
&\leq \norm{S \eta}_{L^2(\Gamma)} + \norm{g_\Gamma}_{L^2(\Gamma)} \\
&\leq C \norm{\eta}_{L^2(\Omega)} +  \norm{g_\Gamma}_{L^2(\Gamma)} \\
&\leq C(\norm{\eta}_{L^2(\Omega)}).
\end{align*}
Similarly using assumption (\ref{eqn:ygcond}) and the supremum norm error estimate (\ref{eqn:errinf}) we get
\begin{align}
\norm{w_\sigma^l}_{L^2(\Gamma)} &= \norm{S_h \eta - g_{\Gamma,\sigma}}_{L^2(\Gamma_\sigma)} \nonumber \\
&\leq \norm{S_h \eta}_{L^2(\Gamma_\sigma)} + \norm{g_{\Gamma,\sigma}^l}_{L^2(\Gamma)} \nonumber \\
&\leq C (\norm{S_h \eta}_\infty + 1) \nonumber \\
&\leq C(\norm{\eta}_{L^2(\Omega)}). \label{eqn:fiop}
\end{align}
%For $\eta \in H^1(\Omega)$ we have by elliptic regularity that
%\[
%\norm{S \eta}_{H^3(\Omega)} \leq C \norm{\eta}_{H^1(\Omega)}
%\]
%(see e.g.\ Theorem~2 in Section~6.3 in \cite{Evans}). Then since $H^3(\Omega)$ is continuously embedded in $W^{1,\infty}(\Omega)$ we have
%\[
%\norm{S\eta}_{W^{1,\infty}(\Omega)} \leq C \norm{\eta}_{H^1(\Omega)}.
%\]
Using (\ref{eqn:boundpoint2}), assumption (\ref{eqn:ygcond}) and the estimate (\ref{eqn:jjscov}) we get
\begin{align*}
\norm{w-w_\sigma^l}_{L^2(\Gamma)} &\leq \norm{S \eta - (S_h \eta)^l }_{L^2(\Gamma)} + \norm{ g_\Gamma - g_{\Gamma,\sigma}^l}_{L^2(\Gamma)} \\
&\leq \norm{S \eta - (S \eta)^l}_{L^2(\Gamma)} + \norm{(S \eta)^l - (S_h \eta)^l }_{L^2(\Gamma)} + C\sigma^2 \\
&\leq C (\sigma^2\norm{S \eta}_{W^{1,\infty}(\Omega)} + \norm{S \eta - S_h \eta}_\infty + \sigma^2) \\
&\leq C (\sigma^2 \norm{\eta}_{H^1(\Omega)} + \norm{S \eta - S_h \eta}_\infty + \sigma^2) \\
&\leq C(\norm{\eta}_{H^1(\Omega)})  (\norm{S \eta - S_h \eta}_\infty + \sigma^2).
\end{align*}
Combining these estimates with (\ref{eqn:lkjlkj}), the bound for the first term in the splitting becomes
\[
\abs{ \norm{w}^2_{L^2(\Gamma)} - \norm{w_\sigma^l}^2_{L^2(\Gamma)} } \leq C( \norm{\eta}_{H^1(\Omega)}) (\norm{S \eta - S_h \eta}_\infty + \sigma^2).
\]
We can bound the second term in the splitting using assumption (\ref{eqn:mcond}) and the estimate (\ref{eqn:fiop}). This completes the proof.
\end{proof}

\begin{remark}
\label{rem:fghjkl}
It is now clear why our assumptions involve $\sigma^2$ bounds as opposed to some other power. When $\sigma=h$, this rate of convergence will not dominate the $h^{2-\eps}$ supremum norm error estimate (\ref{eqn:lineinfest}) (for $n=2$), which we will use to bound $\norm{ S \eta - S_h \eta}_\infty$ on the right hand side of Lemma~\ref{lem:ndiff}. 

Note that if we take $\Gamma_\sigma=\Gamma$, $m_\sigma=m$ and $g_{\Gamma,\sigma}=g_\Gamma$ then the assumptions are trivially satisfied. We will later see that there are nontrivial definitions based on polyhedral interpolations of $\Gamma$ that satisfy the assumptions. If we use polyhedral approximations of $\Gamma$ that are not interpolating then these assumptions may not be satisfied.
\end{remark}

We are ready to use the approach from \cite{Deckelnick2007} and \cite{Leykekhman2013} to prove the following error estimate.

\begin{theorem}
\label{thm:mainline}
Suppose Assumption~\ref{ass:ass} holds. Let $u$ and $u_h$ be solutions of (\ref{eqn:controlprobline}) and (\ref{eqn:controlprob2}) respectively. If $\sigma \leq C h^{\frac{n}{4}}$ and $h$ is sufficiently small then for any $\eps>0$, 
\begin{equation*}
\norm { {u}-{u}_h }_{L^2(\Omega)} \leq  C \Bigg( \sigma + \begin{cases} C(\eps) h^{1-\eps} &n=2,\\
h^{\frac{3}{4}} &n=3 \\
\end{cases} \Bigg )
\end{equation*}
with $C$ independent of $\sigma$ and $h$.
\end{theorem}
\begin{proof}
Fix $\sigma_0$, $h_0 \in \R$ and suppose throughout this proof that $0<\sigma<\sigma_0$ and $0<h<h_0$.

First observe that
\begin{align} \label{eqn:jcomp}
\hat{J}(u_h)-\hat{J}(u) =& \frac{1}{2}\norm{Su_h-Su}^2_{L^2(\Gamma)} + \frac{\nu}{2} \norm{ u_h-u }^2_{L^2(\Omega)} \nonumber \\ &+ (Su_h-Su,Su-g_{\Gamma,\sigma})_{L^2(\Gamma)} + \nu (u,u_h-u) \nonumber \\
\geq& \frac{1}{2} \norm{Su_h-Su}^2_{L^2(\Gamma)} + \frac{\nu}{2} \norm{ u_h-u }^2_{L^2(\Omega)},
\end{align}
since the optimality conditions imply that
\[
(Su_h-Su,Su-g_{\Gamma,\sigma})_{L^2(\Gamma)} = a(Su_h-Su,p)=(u_h-u,p)\geq-\nu (u_h-u,u).
\]
Similarly
\begin{align} \label{eqn:jhcomp}
\hat{J}_h(u)-\hat{J}_h(u_h) \geq \frac{1}{2} \norm{S_h u-S_h u_h}^2_{L^2(\Gamma_\sigma)} + \frac{\nu}{2} \norm{ u-u_h }^2_{L^2(\Omega)}.
\end{align}
Adding these two relations we get\begin{align}
\nu \norm { {u}-{u}_h }_{L^2(\Omega)}^2 &\leq \hat{J}({u}_h) - \hat{J}({u}) -
  \hat{J}_h({u}_h) + \hat{J}_h({u}) \nonumber \\
&\leq \abs{   \hat{J}({u}) - \hat{J}_h({u}) } + \abs{ \hat{J}({u}_h) - \hat{J}_h({u}_h)}. \label{eqn:mkmk}
\end{align}
Lemma \ref{lem:ndiff} gives the estimate
\begin{align*}
\abs{ \hat{J}({u}) - \hat{J}_h({u}) } =& \abs{ \norm{S u-g_\Gamma}^2_{L^2(\Gamma)}-\norm{S_h u-g_{\Gamma,\sigma}}_{\Gamma_\sigma}^2 } \\
\leq& C(\norm{u}_{H^1(\Omega)}) \left ( \norm{S u-S_h u }_{\infty} + \sigma^2 \right ) \\
\leq& C \left ( \norm{S u-S_h u }_{\infty} + \sigma^2 \right )
\end{align*}
with $C$ independent of $\sigma$ and $h$. Now using the supremum norm error estimate (\ref{eqn:lineinfest}) we get that for any $\eps>0$,
\begin{equation}
\abs{ \hat{J}({u}) - \hat{J}_h({u}) } \leq C(\norm{u}_{H^1(\Omega)}) \Bigg(  \sigma^2 + \begin{cases} C(\eps) h^{2-\eps}\norm{u}_{H^1(\Omega)} &n=2,\\
h^{\frac{3}{2}}\norm{u}_{H^1(\Omega)} &n=3 \\
\end{cases} \Bigg ). \label{eqn:asdfg}
\end{equation}
The same approach gives the estimate
\begin{equation}
\abs{ \hat{J}({u}_h) - \hat{J}_h({u}_h)} \leq C(\norm{u_h}_{H^1(\Omega)}) \Bigg( \sigma^2 + \begin{cases} C(\eps) h^{2-\eps}\norm{u_h}_{H^1(\Omega)} &n=2,\\
h^{\frac{3}{2}}\norm{u_h}_{H^1(\Omega)} &n=3 \\
\end{cases} \Bigg ), \label{eqn:lkjhg}
\end{equation}
where the $\norm{u_h}_{H^1(\Omega)}$ term comes from using the supremum norm error estimate (\ref{eqn:lineinfest}). If we take $\sigma \leq Ch^{\frac{n}{4}}$ then by Lemma~\ref{lem:h01bound} we can bound $\norm{u_h}_{H^1(\Omega)}$ independently of $h$.

Combining (\ref{eqn:asdfg}) and (\ref{eqn:lkjhg}) with (\ref{eqn:mkmk}) completes the proof.
\end{proof}

\begin{remark}
We can compare this error estimate to those for analogous discretisations of the standard optimal control problem with an $L^2(\Omega)$ fidelity term and the point control problem considered in \cite{BrettPoint}.
\begin{itemize}
\item Standard control problem from \cite{Hinze2005}:
\begin{equation*}
\norm { {u}-{u}_h }_{L^2(\Omega)} \leq C h^2.
\end{equation*}
\item {Point control problem from \cite{BrettPoint}:} For any $\eps>0$, 
\begin{equation*}
\norm { {u}-{u}_h }_{L^2(\Omega)} \leq C \begin{cases} h^{1-\eps} &n=2,\\
h^{\frac{1}{2}-\eps} &n=3. \\
\end{cases}
\end{equation*}
\item {Control on surface problem (\ref{eqn:controlprobline}) discretised by (\ref{eqn:controlprob2}):} With $\sigma=h$, for any $\eps>0$, 
\begin{equation*}
\norm { {u}-{u}_h }_{L^2(\Omega)} \leq \begin{cases} C(\eps) h^{1-\eps} &n=2,\\
C h^{\frac{3}{4}} &n=3. \\
\end{cases}
\end{equation*}
\end{itemize}
\end{remark}

\subsection{Example definitions of $\Gamma_\sigma$, $m_\sigma$ and $g_{\Gamma,\sigma}$}

So far we have just stated properties that $\Gamma_\sigma$, $m_\sigma$ and $g_{\Gamma,\sigma}$ must have in order for Theorem~\ref{thm:mainline} to hold for the discrete problem (\ref{eqn:controlprobline}). We now give some definitions for these quantities that satisfy all the required properties. Different definitions will lead to discrete problems that are easier or harder to solve, and so the definitions we use in practice will depend on $\Gamma$ and $g_\Gamma$.
\subsubsection{Method 1}

Take the following definitions for $\Gamma_\sigma$, $m_\sigma$ and $g_{\Gamma,\sigma}$ in the discrete problem (\ref{eqn:controlprobline}):
\begin{itemize}
\item $\Gamma_\sigma := \Gamma$ i.e.\ do not approximate $\Gamma$.
\item $m_\sigma(w_\sigma,z_\sigma) := \int_\Gamma w_\sigma z_\sigma \rd A$. This trivially satisfies assumption (\ref{eqn:mcond}) since $w^l=w$ for $w \in L^2(\Gamma)$. 
\item $g_{\Gamma,\sigma}:=g_\sigma$. This trivially satisfies assumption (\ref{eqn:ygcond}).
\end{itemize}
Theorem~\ref{thm:mainline} holds since all the assumptions are satisfied.

We would typically take these choices when $\Gamma$ and $g_{\Gamma}$ have simple forms. For example, perhaps when $\Gamma$ is a straight line or circle and $g_\Gamma$ is piecewise constant function. In this case the integrals over $\Gamma$ of products of discrete functions and $g_\Gamma$ may be easy to compute. This would allow us to implement the numerical method described in Section~\ref{sec:num} exactly.

\begin{remark}
\label{rem:lnkpt}
In practice computing the required integrals over $\Gamma$ will be difficult, even when $\Gamma$ and $g_\Gamma$ are simple. One way to handle this is to use a quadrature in the implementation. See Section~\ref{sec:linkpoint} for a related discussion.
\end{remark}

\subsubsection{Method 2}

\label{ex:compg}

Suppose $g_\Gamma \in H^2(\Gamma)$ and take the following definitions for $\Gamma_\sigma$, $m_\sigma$ and $g_{\Gamma,\sigma}$ in the discrete problem (\ref{eqn:controlprobline}):
\begin{itemize}
\item Let each $\Gamma_\sigma$ consist of a union of finitely many closed $(n-1)$-simplices whose vertices lie on $\Gamma$ and form a conforming, shape regular triangulation $E_\sigma$ of size $\sigma$. By this we mean that $\sigma = {\max}_{E \in E_\sigma} \sigma(E)$ and for each element $E \in E_\sigma$ the quantity
\[
\max_{E \in E_\sigma} \kappa(E), \quad \kappa(E) := \frac{\sigma(E)}{\rho(E)}
\]
is uniformly bounded independently of $\sigma$. Here $\sigma(E)$ denotes the diameter of $E$ and $\rho(E)$ denotes the diameter of the largest ball contained in $E$. %We also require that it satisfies the requirements in Section~\ref{sec:disc}. 

Let $\delta_\sigma$ denote the quotient between the smooth and discrete surface measures $\rd A$ on $\Gamma$ and $\rd A_\sigma$ on $\Gamma_\sigma$ i.e.\ $\delta_\sigma$ is defined by $\delta_\sigma \rd A_\sigma = \rd A$ and
\begin{align}
\int_{\Gamma_\sigma} w_\sigma \rd A_\sigma = \int_\Gamma w_\sigma^l \frac{1}{\delta_\sigma} \rd A \quad  \forall w_\sigma \in L^2(\Gamma_\sigma). \label{eqn:move}
\end{align}

For $\Gamma_\sigma$ defined as above, Lemma 4.1 in \cite{Dziuk2013} gives
\begin{equation}
\norm{1-\frac{1}{\delta_\sigma} }_{L^\infty(\Gamma)} \leq C \sigma^2, \label{eqn:errarea}
\end{equation}
and Lemma 4.2 in \cite{Dziuk2013} gives
\[
\norm{w_\sigma^l}_{L^2(\Gamma)} \leq C \norm{w_\sigma}_{L^2(\Gamma_\sigma)} \quad \forall w_\sigma \in L^2(\Gamma_\sigma)
\]
with $C$ independent of $\sigma$ and $w_\sigma$.

\item $m_\sigma(w_\sigma,z_\sigma) := \int_{\Gamma_\sigma} w_\sigma z_\sigma \rd A_\sigma$. Assumption (\ref{eqn:mcond}) holds, since (\ref{eqn:move}) and (\ref{eqn:errarea}) give that for $w_\sigma,z_\sigma \in L^2(\Gamma_\sigma)$,
\begin{align*}
\abs{ (w_\sigma^l,z_\sigma^l)-m_\sigma(w_\sigma,z_\sigma) } =& \abs{ \int_\Gamma w^l_\sigma z^l_\sigma \Big (1-\frac{1}{\delta_\sigma} \Big ) \rd A } \\ 
\leq& \norm{1-\frac{1}{\delta_\sigma}}_{L^\infty(\Gamma)} \norm{w^l_\sigma}_{L^2(\Gamma)} \norm{z^l_\sigma}_{L^2(\Gamma)} \\
\leq& C\sigma^2 \norm{w^l_\sigma}_{L^2(\Gamma)} \norm{z^l_\sigma}_{L^2(\Gamma)}. 
\end{align*}
 \item $g_{\Gamma,\sigma} := I_\sigma g_\Gamma$, where $I_\sigma$ is the Lagrange interpolation of $g_\Gamma \in H^2(\Gamma)$ onto 
\begin{equation*}
W_\sigma := \{ w_\sigma \in C(\Gamma_\sigma) : w_\sigma|_E \in P_1(E) \text{ for all } E \in E_\sigma  \},
\end{equation*}
the space of piecewise affine finite elements on $E_\sigma$. In particular, $I_\sigma(w) := (\tilde{I}_\sigma w^{-l})$ where $\tilde{I}_\sigma : C(\Gamma_\sigma) \to W_\sigma$ is the Lagrange interpolation operator. By $I_\sigma^l w$ denote $(I_\sigma w)^l$, then for $w \in H^2(\Gamma) \subset C(\Gamma)$ we have
\begin{equation*}
%\label{eqn:errint}
\norm{ w-I_\sigma^l w }_{L^2(\Gamma)} \leq C \sigma^2 \norm{ w }_{H^2(\Gamma)}
\end{equation*}
(see \cite{Dziuk1988} and \cite{Demlow2009}). So assumption (\ref{eqn:ygcond}) is satisfied if $g_\Gamma \in H^2(\Gamma)$.
\end{itemize}
Since all the assumptions are satisfied, Theorem~\ref{thm:mainline} holds. 

We may want to use these definitions of $\Gamma_\sigma$, $m_\sigma$ and $g_{\Gamma,\sigma}$ if $\Gamma$ has a complicated form. In this case it is likely to be hard to calculate integrals over $\Gamma$, which are required by our numerical method (described in Section~\ref{sec:num}). By approximating $\Gamma$ with a polygonal or polyhedral $\Gamma_\sigma$ we only need to compute integrals over straight lines or triangles, which is easier. Note that even if $g_\Gamma$ is quite simple, a complicated $\Gamma$ means that $g_{\Gamma}^l$ could be complicated. This is why we also define $g_{\Gamma,h}$ to be the above piecewise affine interpolation of $g_\Gamma$. Then the surface integrals that are needed for our numerical method simplify to integrals of products of piecewise linear functions over flat surfaces. These are fairly straightforward to calculate and implement.

\begin{remark}
There are a few natural approaches to defining an interpolating polygonal or polyhedral $\Gamma_\sigma$ (see Figure~\ref{fig:examgrids}). These different approaches lead to different challenges. For our numerics we will use approach (c) in the figure, which ensures $\Gamma_\sigma$ coincides with edges (for $n=2$) of $T_h$. This simplifies the calculation of integrals over $\Gamma_\sigma$, but constructing a suitable $T_h$ may be hard. It also effectively forces $\sigma=h$. 
\end{remark}

\begin{remark}
Theorem~\ref{thm:mainline} says that for $n=3$ we could take $\sigma=h^{\frac{3}{4}}$ without dominating the error from the discretisation of the state. We could make use of this if we were to instead use approach (a) in Figure~\ref{fig:examgrids}.
\end{remark}

\begin{figure}
\centering
\subfigure[Here we take an arbitrary interpolation of $\Gamma$. This does not have any relation to the triangulation $T_h$, so calculating integrals over $\Gamma_\sigma$ may be tricky.]{
\includegraphics[width=0.45\textwidth]{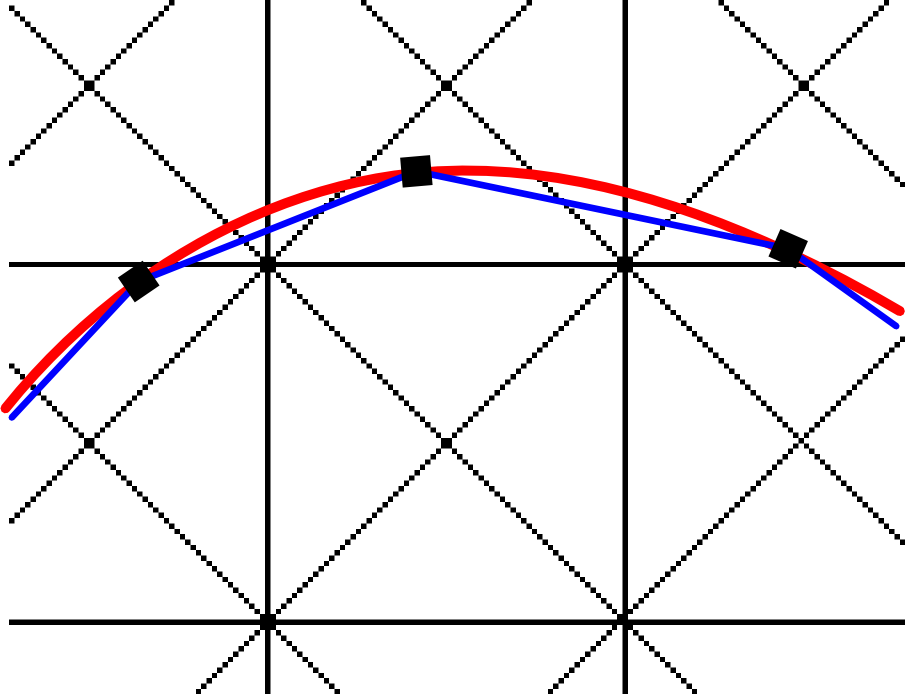}
\label{fig:examgrids1}
}
\subfigure[Here we construct the interpolation of $\Gamma$ using the triangulation $T_h$; the beginning and end of segments of $\Gamma_\sigma$ are given by the points where $\Gamma$ intersects the edges of $T_h$. This makes calculating integrals over $\Gamma_\sigma$ easier.
]{
\includegraphics[width=0.45\textwidth]{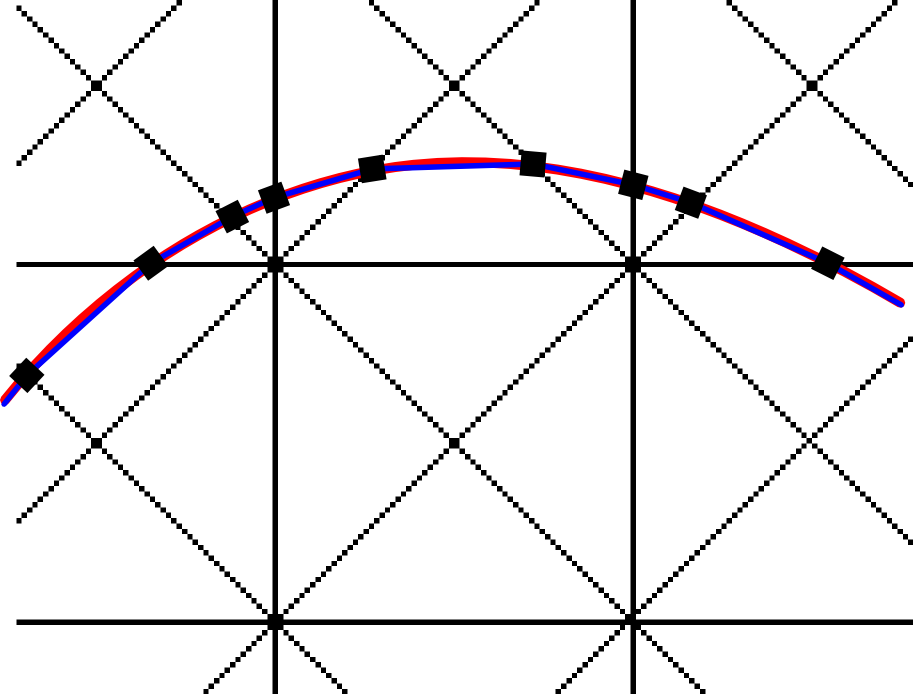}
\label{fig:examgrids2}
}
\subfigure[Here $\Gamma_\sigma$ is chosen first then $T_h$ is constructed to contain the segments of $\Gamma_\sigma$ as edges. This leads to the easiest calculation of integrals over $\Gamma_\sigma$, but constructing $T_h$ may be hard.]{
\includegraphics[width=0.45\textwidth]{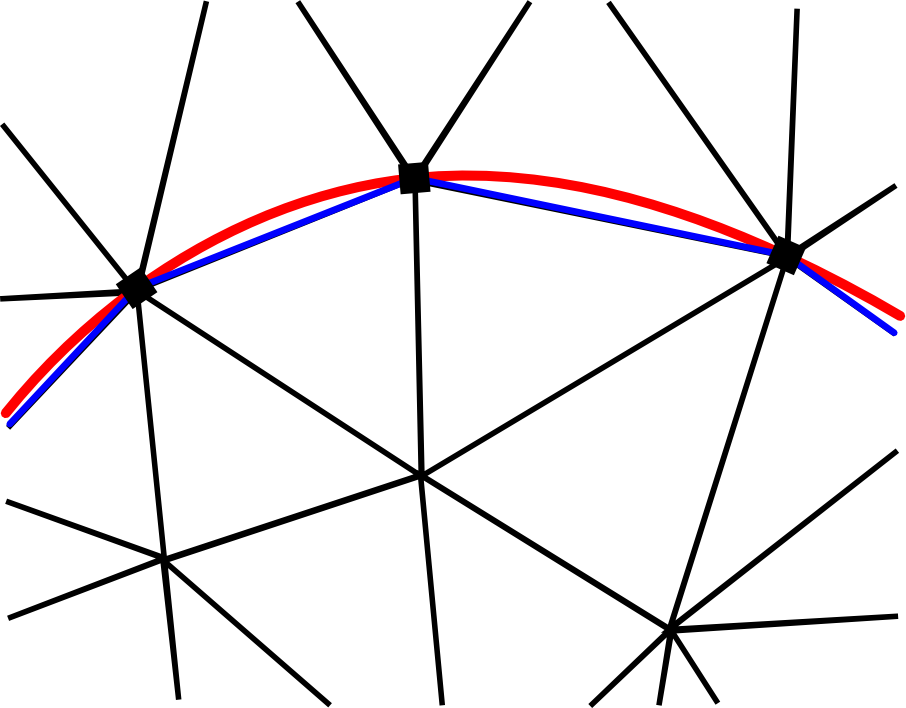}
\label{fig:examgrids3}
}
\caption[Examples of approximating hypersurfaces.]{An illustration of different constructions of $T_h$ and polygonal $\Gamma_\sigma$ for $n=2$. The black lines mark the triangulation $T_h$, the red curve is $\Gamma$ and the blue curve is $\Gamma_\sigma$. The square markers indicate the beginning and end of segments of $\Gamma_\sigma$. \label{fig:examgrids}}
\end{figure}

\subsection{Link to optimal control at points}
\label{sec:linkpoint}

Method~2 can be thought of as using a quadrature to approximate Method~1. Note that 
\[
w_\sigma := S_h \eta|_{\Gamma_\sigma} - I_h g_\Gamma \quad \forall \eta \in L^2(\Omega)
\]
is piecewise linear on $\Gamma_\sigma$. Therefore
\[
\norm{ w_\sigma }^2_{\Gamma_\sigma} = \int_{\Gamma_\sigma} w_\sigma^2  \rd A_h
\]
corresponds to integrating a piecewise quadratic function on $\Gamma_\sigma$. This can be computed exactly with a weighted sum of point evaluations. In particular, Method~2 can be equivalently written as a discrete point control problem:
\begin{equation*}
%\label{eqn:qquadmin}
\min \frac{1}{2} \sum_{w \in I} \kappa_\omega \abs{ S_h \eta(\omega)-g_\omega}^2 + \frac{\nu}{2} \norm{\eta}^2_{L^2(\Omega)} \text{ over } \eta \in L^2(\Omega),
\end{equation*}
where $g_\omega := g_\Gamma(\omega)$, the set $I$ contains points in $\Gamma$, and $\kappa_\omega$ are weights. If we construct a triangulation that contains $\Gamma_\sigma$ as edges (i.e.\ use approach $(c)$ in Figure~\ref{fig:examgrids}) then $\abs{I}=O(\frac{1}{h})$ and the $\kappa_\omega$ are $O(h)$.

As Theorem~\ref{thm:mainline} holds using Method~2, we have provided an example of solutions to discrete point control problems that converge to the solution of a surface control problem.

\begin{remark}
\label{rem:weight}
We could also consider a weighted fidelity term for the surface control problem i.e.\ replace the fidelity term in (\ref{eqn:obj}) by
\[
\frac{1}{2} \int_\Gamma w (y-g_\Gamma)^2 \rd A,
\]
where $w \in L^\infty(\Omega)$ and $w \geq 0$. After the obvious modifications, all the results proved in this paper would still hold.
\end{remark}

\section{Numerical results}
\label{sec:num}

In this section we describe the numerical method we use to solve (\ref{eqn:controlprob2}) and show that the error estimate from Theorem~\ref{thm:mainline} for $n=2$ is observed in practice.

\subsection{Numerical method}
\label{sec:nummeth}

The numerical method is the same as the one described in \cite{BrettPoint} and \cite{BrettThesis} but for a discrete problem without a forcing term $f$ and the point evaluation term replaced by a surface integral term. If $u_h$ solves (\ref{eqn:controlprob2}), then by substituting $u_h = \mathbb{P}_{[a,b]}(-\frac{1}{\nu}p_h)$ we get that the state $y_h :=S_h u_h \in V_h$ and the adjoint variable $p_h \in V_h$ solve
\begin{equation}
\label{eqn:newtonfunpoint}
\left( \begin{array}{c} 
a(y_h, v_h ) - (-\frac{1}{\nu} p_h + (a + \frac{1}{\nu} p_h )^+ - (-\frac{1}{\nu} p_h-b )^+, v_h ) \\
 a(v_h,p_h) - m_\sigma( y_h|_{\Gamma_\sigma}-g_{\Gamma,\sigma}, v_h|_{\Gamma_\sigma}) \end{array} \right) = 0 \quad \forall v_h \in V_h,
\end{equation}
for all $v_h,w_h \in V_h$. Here $v^+$ denotes the nonnegative part of $v$ i.e.\ $\max(0,v)$. Once this problem has been solved, the $u_h$ solving (\ref{eqn:controlprob2}) can easily be determined from $p_h$ by setting $u_h=\mathbb{P}_{[a,b]}\big (-\frac{1}{\nu} p_h \big)$. 

We use a semismooth Newton method to solve the above system, but we will not describe the algorithm in detail as it follows from only minor modifications to the one in \cite{BrettPoint} and \cite{BrettThesis}. We implemented the algorithm for $n=2$ in the Distributed and Unified Numerics Environment (DUNE) using DUNE-FEM (see \cite{dunegridpaperI:08, dunegridpaperII:08, dunefempaper:10}). This environment has the advantage that once an algorithm has been implemented, it is straightforward to change features of the implementation that would usually be fixed. For solving the linear systems for each iteration of the Newton method we used the biconjugate gradient stabilised method with an incomplete LU factorisation or Gauss-Seidel preconditioner. We do not implement the $n=3$ as this would be more complicated.

Depending on the example we are considering, we may either use Method~1 or Method~2 to choose $\Gamma_\sigma$, $m_\sigma$ and $g_{\Gamma,\sigma}$. When using Method~2 we will use the approach from Figure~\ref{fig:examgrids3} and construct the triangulation $T_h$ from $\Gamma_\sigma$: We first find a polygonal curve $\Gamma_\sigma$ with segments of length $h$ (i.e.\ we take $\sigma=h$), % using an arc length parameterisation of $\Gamma$ (which, if necessary, can be computed numerically to high precision from an arbitrary parameterisation). 
and then use the program Triangle (see \cite{Shewchuk}) to construct a triangulation $T_h$ of size $h$ that contains the segments of $\Gamma_\sigma$ as edges.

\subsection{Examples}

In all our examples we will take $A=-\Laplace$ and $\sigma=h$. We first solve two simple examples on a $\Gamma$ that is a straight line. 

\begin{example} 
\label{ex:ex1point}
$\Omega=(0,1)^2$, $\Gamma = \{ (0.25+0.5t, 0.5) : t \in (0,1) \}$,
$g_\Gamma(x_1,x_2) =\sin(3 \pi x_1)$, $\nu=1e-2$, $b=-a=\infty$.
\end{example}

\begin{example}
\label{ex:ex2}
The same as Example~\ref{ex:ex1point} but with
\[
g_\Gamma(x_1,x_2) = \begin{cases}1 & x <0, \\
-1 & x \geq 0 \\
\end{cases}
\]
and $b=-a=5$.
\end{example}

Example~\ref{ex:ex1point} has a smooth but nonconstant $g_\Gamma$ and no control constraints. Its solution can be seen in Figure~\ref{fig:ypline}. Example~\ref{ex:ex2} has a discontinuous $g_\Gamma$ and active control constraints. Its solution can be seen in Figure~\ref{fig:linecon}. Even for these simple examples the exact solution is not known explicitly, so we compute $L^2(\Omega)$ errors against discrete solutions on fine triangulations to get approximate experimental orders of convergence (EOCs). In particular we use
\begin{equation*}
\mathrm{EOC}_h = \frac{\log (\norm{\tilde{u}-u_{h/2}}_{L^2(\Omega)}/\norm{\tilde{u}-u_h}_{L^2(\Omega)})}{\log 2}
\end{equation*}
with $h_{\text{fine}}=0.00276214$, which corresponds to 263169 DOFs. The approximate EOCs for these examples are in Tables~\ref{tab:eocs} and \ref{tab:linecon}. They agree with the error estimate we proved in Theorem~\ref{thm:mainline} for $n=2$.  We do not verify this error estimate for examples with curved $\Gamma$: With our approach of constructing triangulations $T_h$ that coincide with $\Gamma_\sigma$, the resulting $T_h$ for a small $h$ will not in general be a refinement of a $T_h$ for a larger $h$. This makes it challenging to compute $L^2(\Omega)$ errors.

\begin{figure}
\centering
\subfigure[Illustration of $\Gamma$.]{
\includegraphics[width=0.3\textwidth]{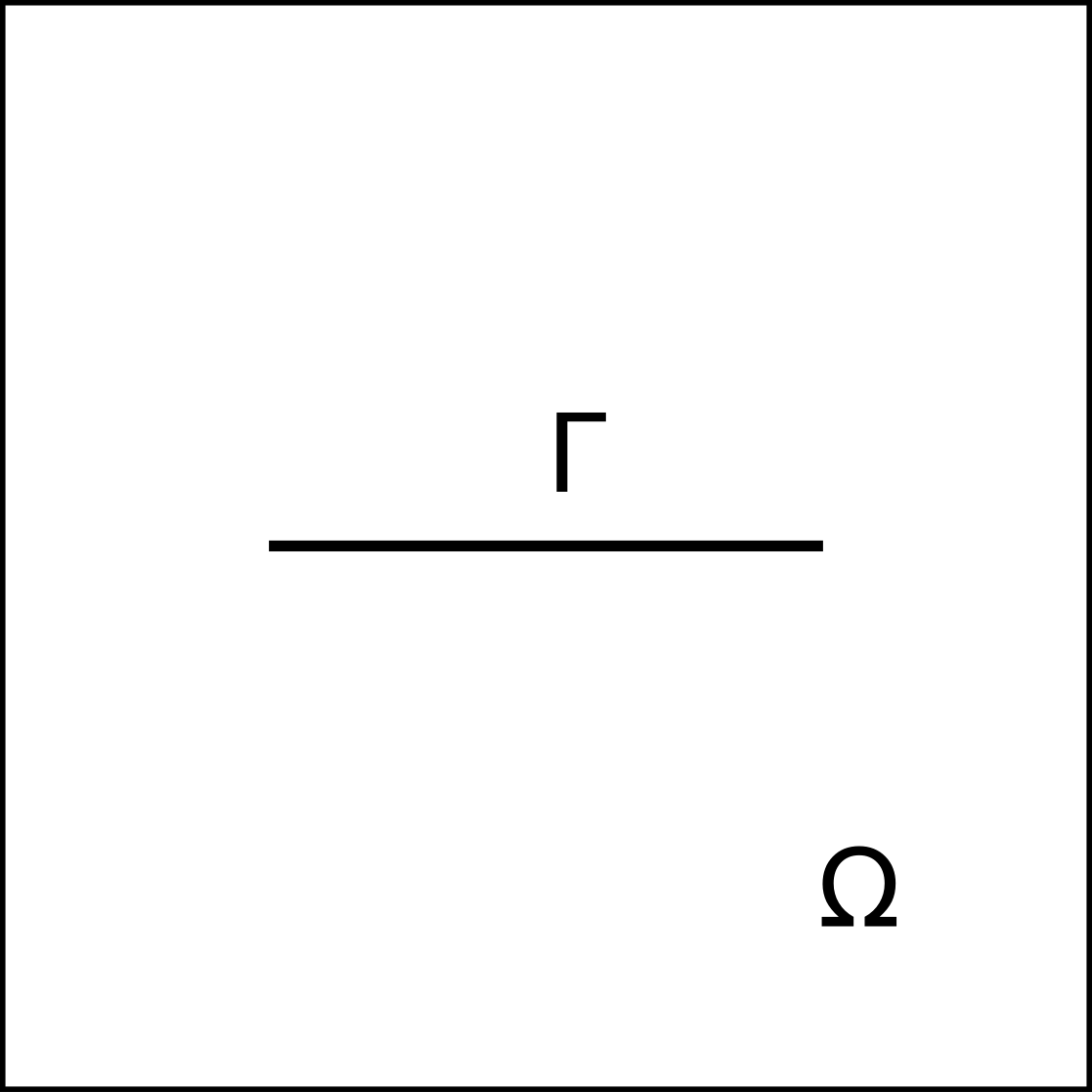}
}
\subfigure[Surface plot of $y_h$.]{
\includegraphics[width=0.4\textwidth]{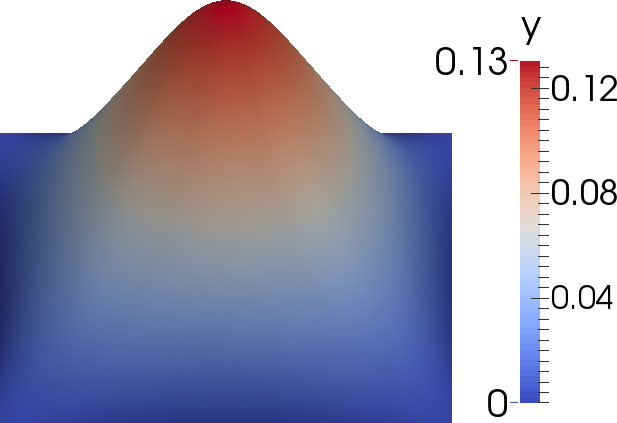}
}
\subfigure[Surface plot of $u_h= -\frac{1}{\nu}p_h$ with the $x_1$-axis at the bottom.]{
\includegraphics[width=0.41\textwidth]{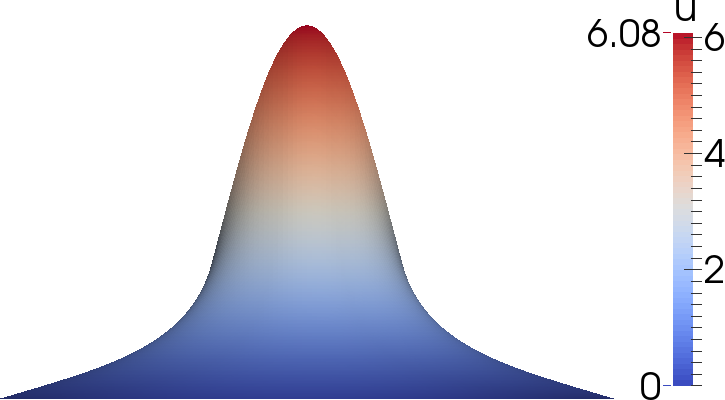}
}
\subfigure[Surface plot of $u_h= -\frac{1}{\nu}p_h$ with the $x_2$-axis at the bottom.]{
\includegraphics[width=0.4\textwidth]{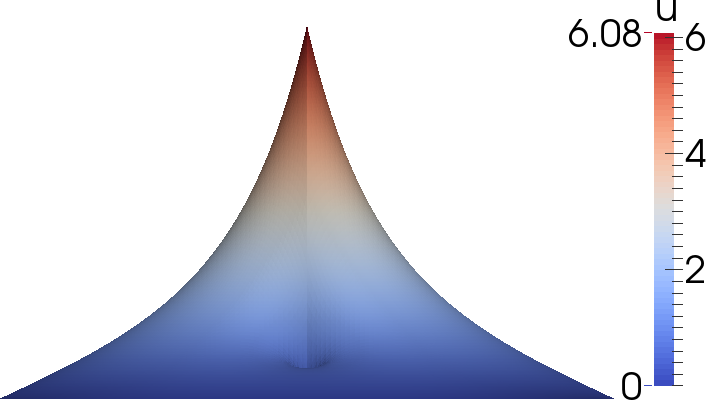}
}
\caption[Solution of a PDE surface control problem (no control constraints).]{The solution to Example~\ref{ex:ex1point}. We use $\Gamma_\sigma$, $m_\sigma$ and $g_{\Gamma,\sigma}$ defined by Method~2, even though we do not need to approximate $\Gamma$, as interpolating $g_\Gamma$ simplifies the implementation.}
\label{fig:ypline}
\end{figure}

\begin{figure}
\centering
\subfigure[Illustration of $\Gamma$.]{
\includegraphics[width=0.3\textwidth]{simpleline}
} 
\subfigure[Surface plot of $y_h$.]{
\includegraphics[width=0.42\textwidth]{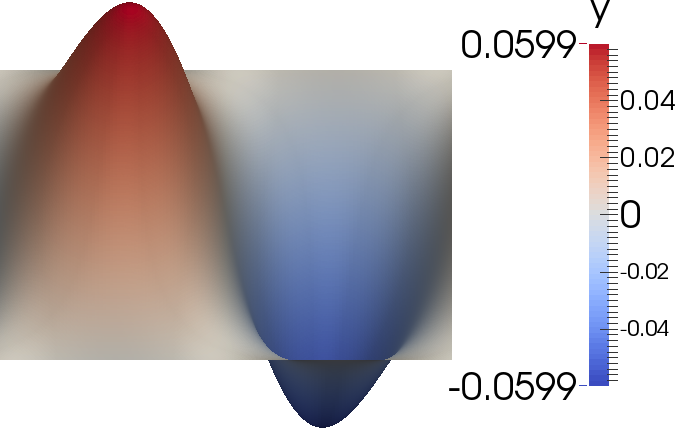}
}
\subfigure[Surface plot of $u_h= \mathbb{P}_{[a,b]}\big( -\frac{1}{\nu}p_h \big )$ with the $x_1$-axis at the bottom.]{
\includegraphics[width=0.4\textwidth]{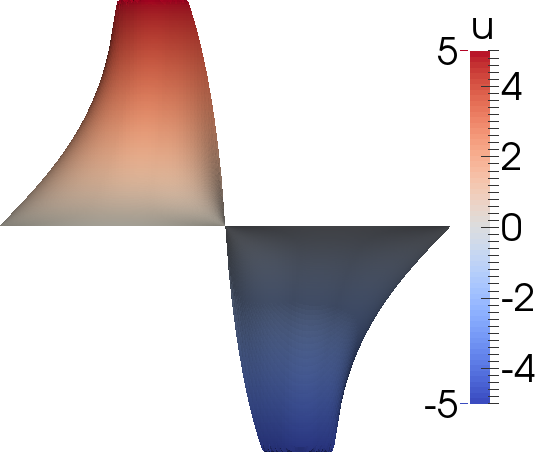}
}
\subfigure[Surface plot of $u_h= \mathbb{P}_{[a,b]}\big( -\frac{1}{\nu}p_h \big )$ with the $x_2$-axis at the bottom.]{
\includegraphics[width=0.4\textwidth]{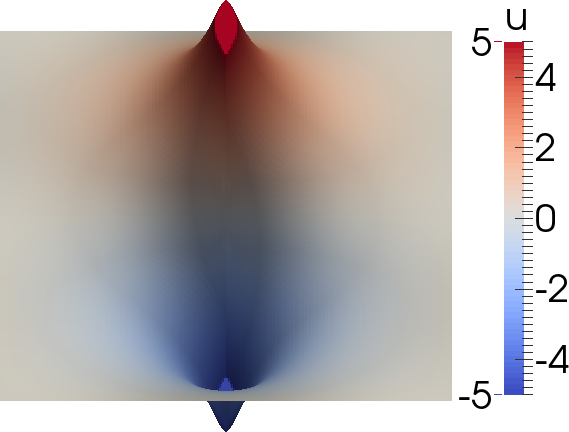}
}
\caption[Solution of a PDE surface control problem (control constraints)]{The solution to Example~\ref{ex:ex2}. We use $\Gamma_\sigma$, $m_\sigma$ and $g_{\Gamma,\sigma}$ defined by Method~1, as it is easy to integrate discrete functions against $g_{\Gamma,\sigma}$ along $\Gamma$. The figure can be interpreted in the same way as Figure~\ref{fig:ypline}.}
\label{fig:linecon}
\end{figure}

\begin{table}
\begin{center}
\begin{tabular}{ c | c | c | c }
\input{convtable}
\end{tabular}
\end{center}
\caption[EOCs for PDE surface control problem (no control constraints).]{EOCs for Example \ref{ex:ex1point} (which has no active control constraints).}
\label{tab:eocs}
\end{table}

\begin{table}
\centering
\begin{tabular}{c|c|c|c}
$h$ & \# DoFs & $\norm{u-u_h}_{L^2(\Omega)}$ & $\text{EOC}_h$ \\ 
% BEGIN RECEIVE ORGTBL linecon
\hline
0.353553 & 25 & 0.991883 & 0 \\
0.176777 & 81 & 0.544039 & 0.86646341 \\
0.0883883 & 289 & 0.292202 & 0.89674110 \\
0.0441942 & 1089 & 0.146281 & 0.99822529 \\
0.0220971 & 4225 & 0.0741029 & 0.98114049 \\
0.0110485 & 16641 & 0.0363584 & 1.0272346 \\
% END RECEIVE ORGTBL linecon
\end{tabular}
\caption[EOCs for PDE surface control problem (control constraints).]{EOCs for Example \ref{ex:ex2} (which has active control constraints).}
\label{tab:linecon}
\end{table}
%
\begin{comment}
 #+ORGTBL: SEND linecon orgtbl-to-latex :splice t :skip 1
|         h |  DoFs |     Error |        EOC |
|-----------+-------+-----------+------------|
|  0.353553 |    25 |  0.991883 |          0 |
|  0.176777 |    81 |  0.544039 | 0.86646341 |
| 0.0883883 |   289 |  0.292202 | 0.89674110 |
| 0.0441942 |  1089 |  0.146281 | 0.99822529 |
| 0.0220971 |  4225 | 0.0741029 | 0.98114049 |
| 0.0110485 | 16641 | 0.0363584 |  1.0272346 |
#+TBLFM: @2$4=0::$4=log(@-1$3/$3)/log(@-1$1/$1)
\end{comment}

In comparison to solutions of point control problems, the solutions of these line control examples appear to have bounded $p$ (and hence also $u$). An interesting feature of the solutions are the ridges in $p_h$ and $u_h$ along $\Gamma$. Observe that in the above examples $y_h|_\Gamma$ does not get close to $g_\Gamma$ because $\nu=1e-2$ is too large, especially when there are control constraints. In the next examples we take $\nu=1e-4$ and observe that we can get close agreement between $ y_h|_{\Gamma}$ and $g_\Gamma$. In the remaining examples the only variable that will change is $\Gamma$.

\begin{example}
\label{ex:ex3}
$\Omega=(0,1)^2$, 
\[
\Gamma = \{ ( 0.5 + 0.327t\sin t, 0.5+0.327t \cos t )  :  t \in (0,3.159) \},
\]
(i.e.\ a spiral), $g_\Gamma=1$, $\nu=1e-4$, $b=-a=\infty$.
\end{example}

%\begin{example}
%\label{ex:ex4}
%The same as Example~\ref{ex:ex3} but with
%\[
%\Gamma = \{ (0.5 + 0.25 \cos(2 \pi t), 0.5 + 0.25 \sin(2 \pi t) : t \in (0,1) \}
%\]
%(i.e.\ a circle). 
%\end{example}

\begin{example}
\label{ex:ex5}
The same as Example~\ref{ex:ex3} but with a multi-component $\Gamma$ having the spoke like structure marked by the black lines in Figure~\ref{fig:spokelines}.
\end{example}

In Example~\ref{ex:ex3} $\Gamma$ is curved. As described in Section~\ref{sec:num},  we first construct a $\Gamma_\sigma$ that interpolates $\Gamma$, then create a triangulation that coincides with $\Gamma_\sigma$. To illustrate this a possible (but coarse) triangulation for the spiral shaped $\Gamma$ from Example~\ref{ex:ex3} is shown in Figure~\ref{fig:domain}. In Example~\ref{ex:ex5} the spoke like $\Gamma$ is formed from a $\Gamma$ consisting of multiple connected components; in particular, 6 open lines originating from the point $(0.5,0.5)$.  % As suggested in Remark~\ref{rem:conn}, some of the theory we have developed does not quite apply to such a $\Gamma$. Nevertheless, our discretisation still makes sense and we expect the solution of this to closely approximate the true solution.

Solutions to Examples~\ref{ex:ex3} and \ref{ex:ex5} can be seen in Figures~\ref{fig:spiral} and \ref{fig:spoke}. These were computed with $h=0.00292967$ and $\# \mathrm{DOFs} \approx 70000$. Observe that for this small value of $\nu=1e-4$, the values of $y_h|_\Gamma$ are close to $g_\Gamma =1$.

\begin{figure}
\centering
\subfigure{
\includegraphics[width=0.32\textwidth]{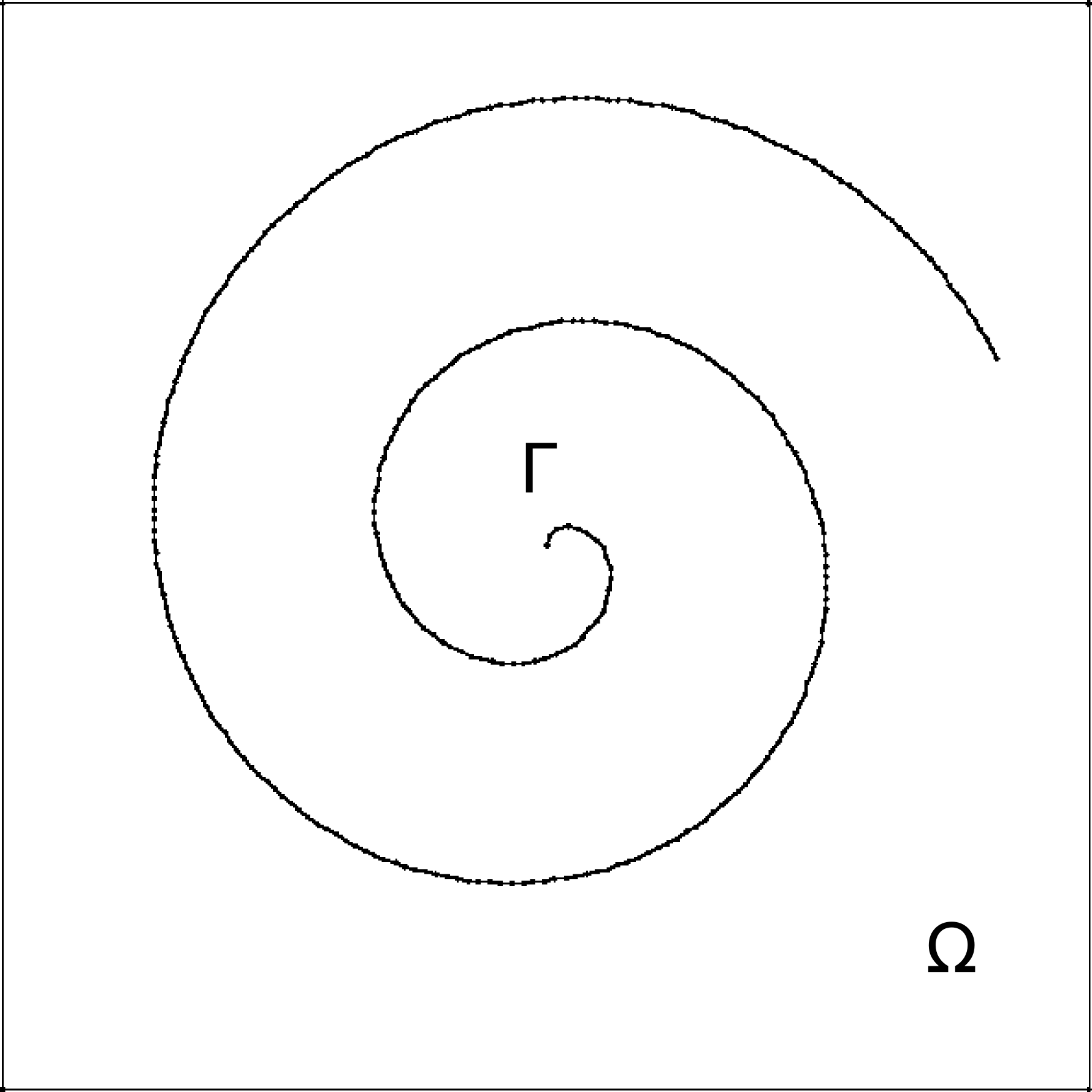}
\label{fig:dom1}
}
\subfigure{
\includegraphics[width=0.325\textwidth]{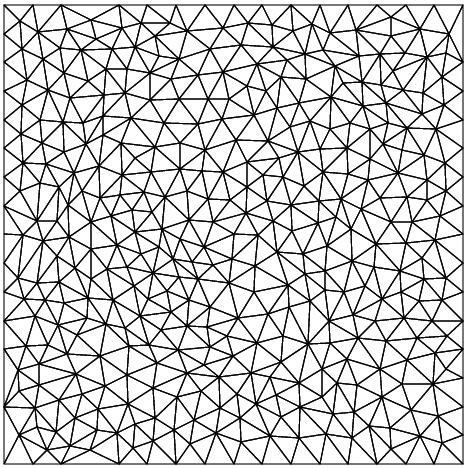}
\label{fig:dom2}
}
\caption[Triangulation containing an approximating hypersurface.]{$\Gamma$ as defined in Example \ref{ex:ex3} and a triangulation whose edges contain an interpolating polygonal approximation of it. }
\label{fig:domain}
\end{figure}

%\begin{figure}
%\centering
%\subfigure[Surface plot of $y_h$.]{
%\includegraphics[width=0.47\textwidth]{figs/circle_y_angle}
%}
%\subfigure[Surface plot of $u_h=-\frac{1}{\nu}p_h$.]{
%\includegraphics[width=0.47\textwidth]{figs/circle_u_angle}
%}
%\subfigure[Colour map of $y_h$.]{
%\includegraphics[width=0.47\textwidth]{figs/circle_y_above}
%\label{fig:jjj}
%}
%\subfigure[Colour map of $u_h=-\frac{1}{\nu}p_h$.]{
%\includegraphics[width=0.47\textwidth]{figs/circle_u_above}
%}
%\caption[Solution of a PDE surface control problem ($\Gamma$ a circle).]{The solution to Example~\ref{ex:ex4}. We use $\Gamma_\sigma$, $m_\sigma$ and $g_{\Gamma,\sigma}$ defined by Method~2. The black curve in Figure~ \ref{fig:jjj} is $\Gamma$ and the dots and numerical values indicate the value of $y_h$ at certain points on $\Gamma$. Not many points are included due to the symmetry of the solution.}
%\label{fig:circle}
%\end{figure} 
 
\begin{figure}
\centering
\subfigure[Surface plot of $y_h$.]{
\includegraphics[width=0.4\textwidth]{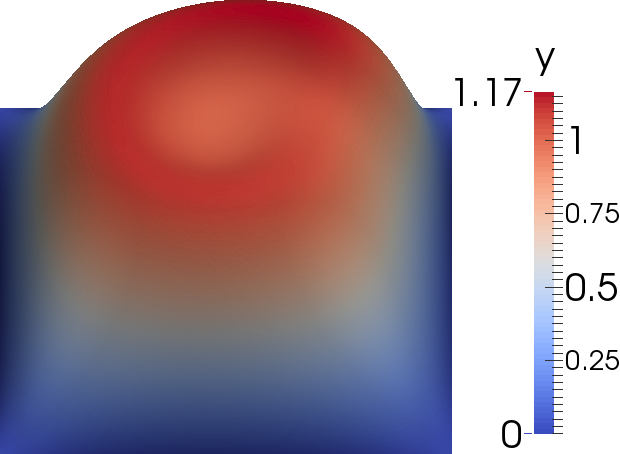}
}
\subfigure[Surface plot of $u_h=-\frac{1}{\nu}p_h$.]{
\includegraphics[width=0.4\textwidth]{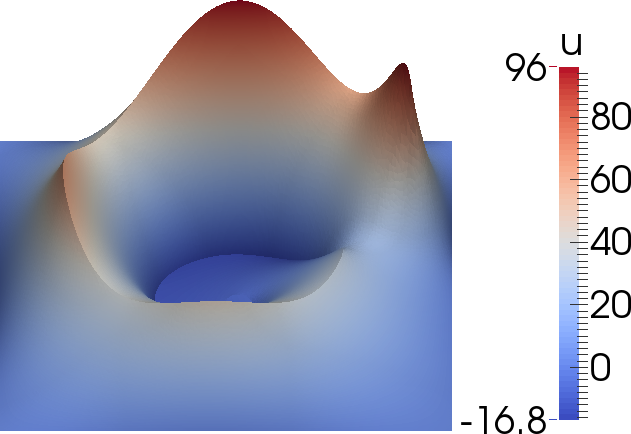}
}
\subfigure[Colour map of $y_h$.]{
\label{fig:ppp}
\includegraphics[width=0.4\textwidth]{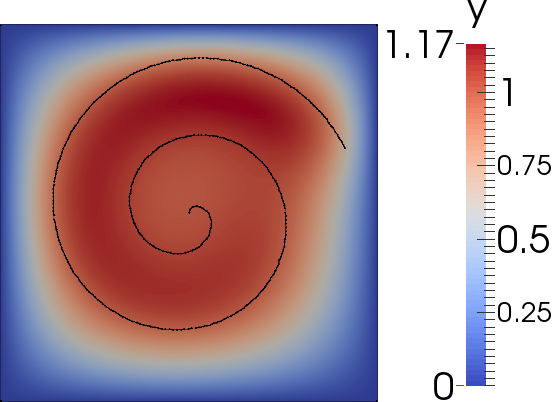}
}
\subfigure[Colour map of $u_h=-\frac{1}{\nu}p_h$.]{
\label{fig:ppp2}
\includegraphics[width=0.4\textwidth]{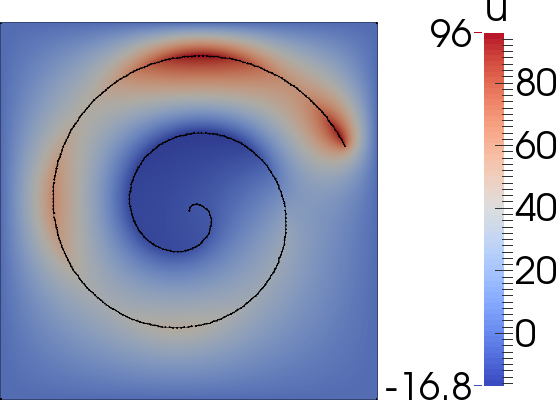}
}
%\subfigure[Slice of $y_h$ along the yellow line in Figure~\ref{fig:ppp}.]{
%\includegraphics[width=0.42\textwidth]{figs/spiral_y_slice2}
%\label{fig:e}
%}
%\subfigure[Slice of $u_h=-\frac{1}{\nu}p_h$ along the yellow line in %Figure~\ref{fig:ppp2}.]{
%\includegraphics[width=0.42\textwidth]{figs/spiral_p_slice2}
%\label{fig:f}
%}
\caption[Solution of a PDE surface control problem ($\Gamma$ a spiral).]{The solution to Example~\ref{ex:ex3}. We use $\Gamma_\sigma$, $m_\sigma$ and $g_{\Gamma,\sigma}$ defined by Method~2. The black curve is $\Gamma$.% In addition the yellow line indicates the slice used to produce Figures~\ref{fig:e} and \ref{fig:f}.
}
\label{fig:spiral}
\end{figure} 

\begin{figure}
\centering
\subfigure[Surface plot of $y_h$.]{
\includegraphics[width=0.4\textwidth]{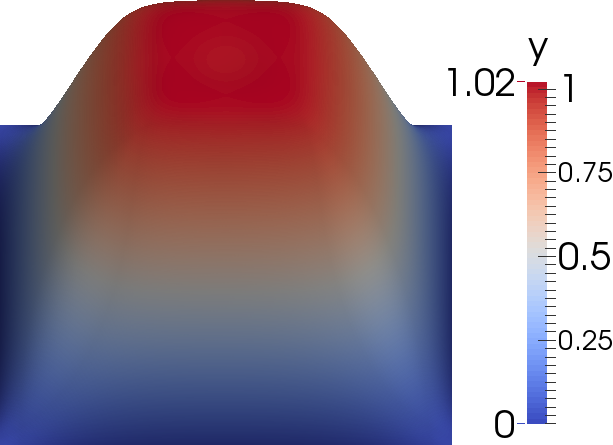}
}
\subfigure[Surface plot of $u_h=-\frac{1}{\nu}p_h$.]{
\includegraphics[width=0.4\textwidth]{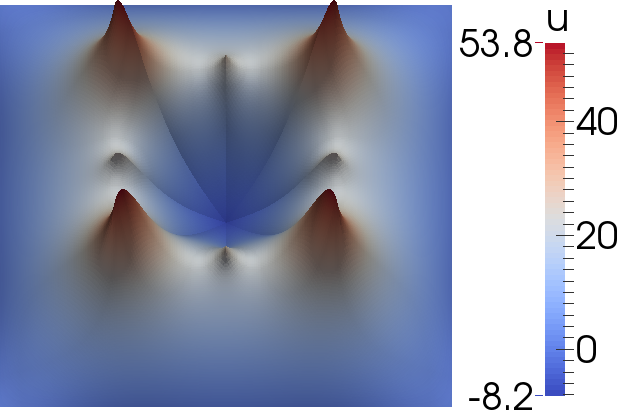}
}
\subfigure[Colour map of $y_h$.]{
\includegraphics[width=0.41\textwidth]{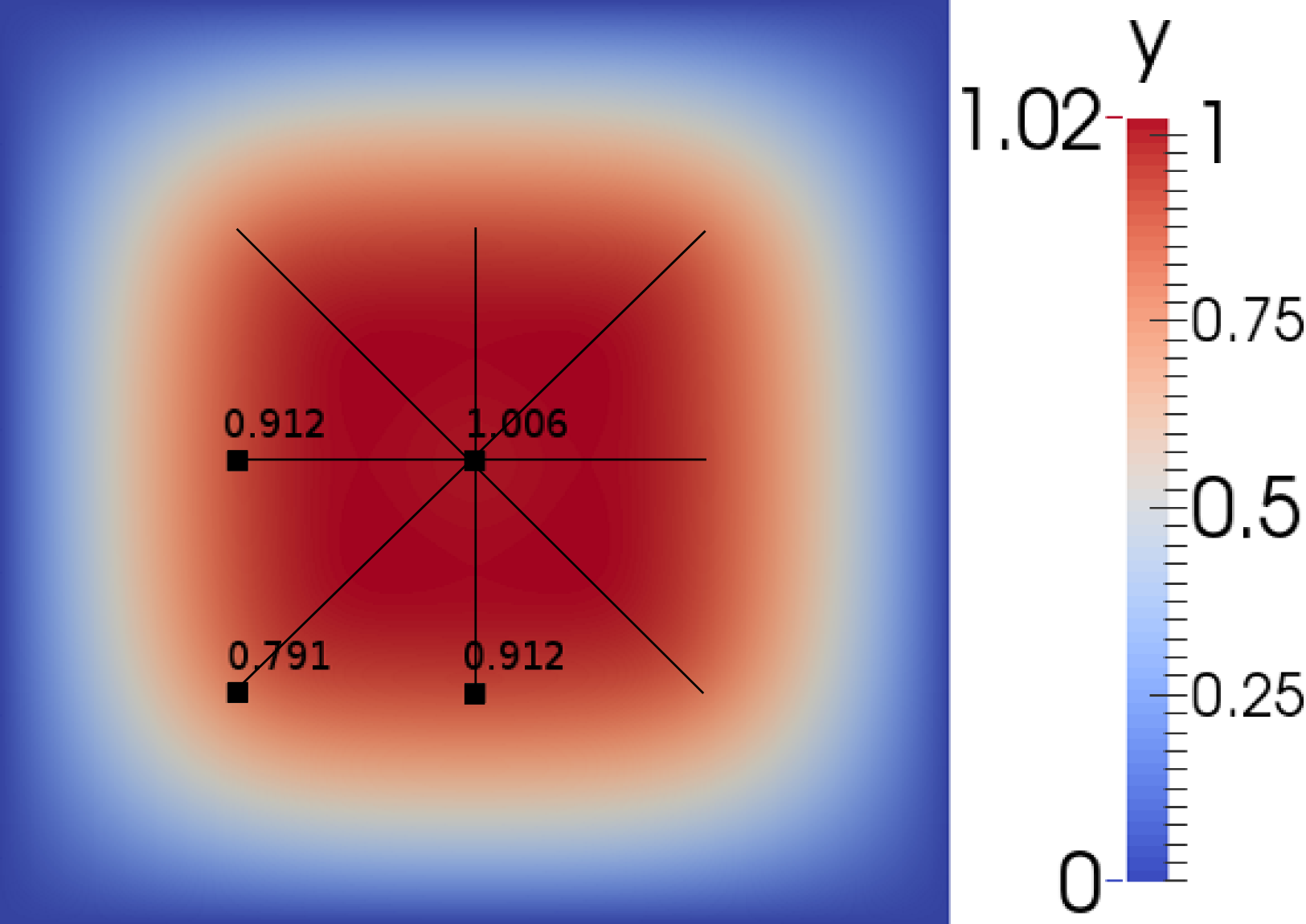}
\label{fig:spokelines}
}
\subfigure[Colour map of $u_h=-\frac{1}{\nu}p_h$.]{
\includegraphics[width=0.4\textwidth]{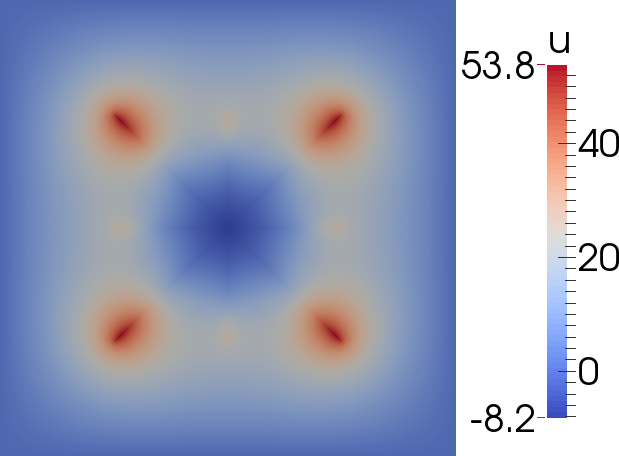}
}
%\subfigure[Slices of $u_h=-\frac{1}{\nu}p_h$ along coloured lines in Figure~\ref{fig:spokelines}. The green curve in Figure~\ref{fig:kkk} has been scaled along the $x$-axis.]{
%\includegraphics[width=0.46\textwidth]{figs/spoke_y_slices3}
%\label{fig:kkk}
%}
\caption[Solution of a PDE surface control problem ($\Gamma$ has spokes).]{The solution to Example~\ref{ex:ex5}. We use $\Gamma_\sigma$, $m_\sigma$ and $g_{\Gamma,\sigma}$ defined by Method~1. The black curve is $\Gamma$ and the dots and numerical values indicate the value of $y_h$ at certain points on $\Gamma$. Not many points are included due to the symmetry of the solution. %Figure~\ref{fig:kkk} shows some slices of $y_h$.
}
\label{fig:spoke}
\end{figure}

\subsection{Comparison to optimal control at points}

\label{sec:knvc}
To finish we compare the solution of the line problem from Example~\ref{ex:ex3} (shown in Figure~\ref{fig:spiral}) with the following point control problem.

\begin{example}
\label{ex:pointspiral}
$\Omega=(0,1)^2$, 
\[
\Gamma = \{ ( 0.5 + 0.327t\sin t, 0.5+0.327t \cos t )  :  t \in (0,3.159) \},
\]
(i.e. a spiral), $I$ is a set of $41$ evenly spaced points along $\Gamma$, $g_\omega=1$ for all $\omega \in I$, $\nu=1e-4$, $b=-a=\infty$.
\end{example}

The theory for such problems is covered in \cite{BrettPoint}. Note that we take the same parameter values as for the line problem except instead of a prescribed function $g_\Gamma = 1$, we have prescribed values of $g_\omega = 1$ at points along $\Gamma$. The solution of this problem can be seen in Figure~\ref{fig:pointspiral}.

We see in Figure~\ref{subfig:comp} that the point problem gets $y_h|_\Gamma$ closer to $1$ than the line problem. However this is at the cost of $\norm{u_h}_{L^2(\Omega)} = 36.5414$ for the point problem compared to $\norm{u_h}_{L^2(\Omega)} = 28.0718$ for the line problem, and a what appears to be unbounded $\norm{u}_\infty$.

\begin{figure}
\centering
%\subfigure[Surface plot of $y_h$.]{
%\includegraphics[width=0.46\textwidth]{figs/pointspiralyangle}
%}
\subfigure[Colour map of $y_h$. The black line is the curve $\Gamma$ along which the points are evenly distributed.]{
\includegraphics[width=0.39\textwidth]{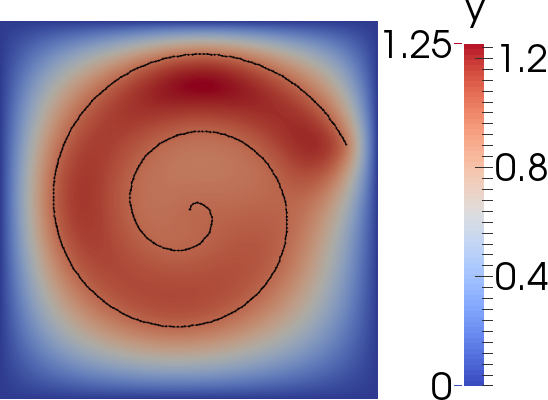}
}
\subfigure[Surface plot of $u_h=-\frac{1}{\nu}p_h$.]{
\includegraphics[width=0.4\textwidth]{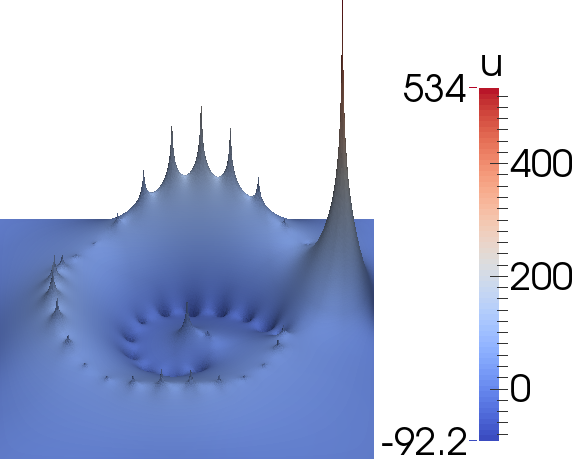}
}
%\subfigure[Colour map of $u_h=-\frac{1}{\nu}p_h$.]{
%\includegraphics[width=0.48\textwidth]{figs/pointspiraluflat}
%}
\subfigure[Comparison of $y_h$ from this point control problem evaluated on $I$, against $y_h$ from the line control problem of Example~\ref{ex:ex3} (shown in Figure~\ref{fig:spiral}) evaluated along $\Gamma$ starting from the centre. %We have parameterised $\Gamma$ by $t$.
]{
\includegraphics[width=0.38\textwidth]{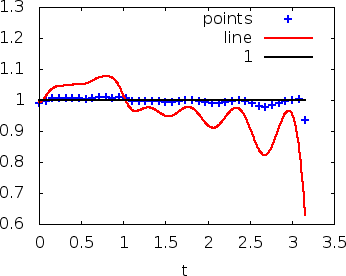}
\label{subfig:comp}
}
% \subfigure[$u$]{
% \includegraphics[width=0.42\textwidth]{figs/spiral_p_slice2}
% \label{fig:f}
% }
\caption[Solution of a PDE point control problem (points on a spiral).]{The solution of Example~\ref{ex:pointspiral}. Compare to Figure~\ref{fig:spiral}. }
\label{fig:pointspiral}
\end{figure} 

Recall our observation from Section~\ref{sec:linkpoint} that solutions of appropriately weighted discrete point control problems converge to the solution of a surface control problem. The points and weights we mentioned %Corollary~\ref{cor:pointconv} resulted from a second order Gaussian quadrature rule, 
arose from Method~2. A simpler approach, which nevertheless works well in practice, is to choose an arbitrary triangulation of size $h$, then take $\lceil \frac{\abs{\Gamma}}{h} \rceil$ (where $\lceil \cdot \rceil$ denotes the ceiling function) evenly spaced points along $\Gamma$ and weight them by $h$. Given an arclength parameterisation of $\Gamma$, it is straightforward to adapt the implementation in \cite{BrettPoint} to do this. 

The solution to the point control problem resulting from this approach is almost indistinguishable to the solution of the line control problem using Method~2, so we do not include it in a figure. A minor difference is that the ridge in $u_h$ is slightly jagged, as the edges of the triangulation do not necessarily align with it, but as $h$ is reduced this effect disappears.

%\begin{figure}
%\centering
%\subfigure[Surface plot of $y_h$.]{
%\includegraphics[width=0.455\textwidth]{figs/quadspiral_y_warp}
%}
%\subfigure[Surface plot of $u_h=-\frac{1}{\nu}p_h$.]{
%\includegraphics[width=0.46\textwidth]{figs/quadspiral_u_warp}
%}
%\caption[PDE point control problem approximating a PDE surface control problem.]{The solution to a point control that appears to closely approximate the solution to Example~\ref{ex:ex3} (see Figure~\ref{fig:spiral}).}
%\label{fig:spiral2}
%\end{figure}

%%% Local Variables: 
%%% mode: latex
%%% TeX-master: "linepaper.tex"
%%% End:  

%% file: convtable
h & DoFs & $\norm{u_h-\tilde{u}}_\Omega$ & EoC \\
\hline
0.353553 & 25 & 0.92096 & - \\
0.176777 & 81 & 0.413261 & 1.1561 \\
0.0883883 & 289 & 0.193377 & 1.0956 \\
0.0441942 & 1089 & 0.0967977 & 0.9984 \\
0.0220971 & 4225 & 0.0482398 & 1.0047 \\
0.0110485 & 16641 & 0.0235625 & 1.0337 \\